\documentclass[ijoc,nonblindrev]{informs3} 

\OneAndAHalfSpacedXII 

\usepackage{natbib}
 \bibpunct[, ]{(}{)}{,}{a}{}{,}%
 \def\bibfont{\small}%

\TheoremsNumberedThrough     

\EquationsNumberedThrough    

\usepackage{amsmath}
\usepackage{amssymb}
\usepackage{placeins}
\usepackage{hhline}
\usepackage{braket}
\usepackage{url}
\usepackage[colorlinks = false,
            linkcolor = blue,
            urlcolor  = blue,
            citecolor = blue,
            anchorcolor = blue]{hyperref}
\usepackage{enumitem}
\usepackage[table,xcdraw]{xcolor}
\usepackage{booktabs}
\usepackage{afterpage}
\usepackage{tikz}
\usepackage{pgfkeys}
\usetikzlibrary{fit}
\usetikzlibrary{shapes.geometric}
\usetikzlibrary{decorations.pathreplacing}
\usetikzlibrary{calc}
\usepackage{breqn}
\usepackage{alphalph,etoolbox}
\patchcmd{\subequations}{\alph{equation}}{\alphalph{\value{equation}}}{}{}

\newcommand{\bsubeq}{\begin{subequations}}
\newcommand{\esubeq}{\end{subequations}}
\newcommand{\BI}{\begin{itemize}}
\newcommand{\EI}{\end{itemize}}
\newcommand{\I}{\item}

\newcommand{\cblack}{\color{black}}

\newcommand{\cred}{\color{black}}

\newcommand{\cblue}{\color{black}}
\newcommand{\dblue}{\color{black}}

\renewcommand{\mc}{\mathcal}

\def\st{{\rm s.t.}}

\usepackage{algorithmic}
\usepackage{algorithm}
\usepackage{graphicx}
\usepackage{enumerate}
\usepackage{subfigure}
\usepackage[normalem]{ulem}

\newcommand{\conv}{{\rm conv}}

\newcommand{\thetau}{\bs{\theta}_{ij}^{\mathbf{u}}}
\newcommand{\thetaM}{\bs{\theta}^{\mathbf{M}}}
\newcommand{\img}{\mathbf{j}}
\newcommand{\thetalb}{\underline{\bs{\theta}}_{ij}}
\newcommand{\thetaub}{\overline{\bs{\theta}}_{ij}}
\newcommand{\vlb}{\bs{\underline{v}}}
\newcommand{\vub}{\bs{\overline{v}}}
\newcommand{\vsigma}{\bs{v}^{\sigma}}
\newcommand{\phiij}{\bs{\phi}_{ij}}
\newcommand{\deltaij}{\bs{\delta}_{ij}}
\newcommand{\bcij}{\bs{b}^c_{ij}}
\newcommand{\bsij}{\bs{b}^s_{ij}}
\newcommand{\gcij}{\bs{g}^c_{ij}}
\newcommand{\gij}{\bs{g}_{ij}}
\newcommand{\gsij}{\bs{g}^s_{ij}}
\newcommand{\Ycij}{\bs{Y}_{ij}^c}
\newcommand{\Ycji}{\bs{Y}_{ji}^c}
\newcommand{\Yij}{\bs{Y}_{ij}}
\newcommand{\Ysij}{\bs{Y}_{ij}^s}
\newcommand{\Tij}{\bs{T}_{ij}}
\newcommand{\tRij}{\bs{t}^{R}_{ij}}
\newcommand{\tIij}{\bs{t}^{I}_{ij}}
\newcommand{\sijbar}{\overline{\bs{s}}^{a}_{ij}}
\newcommand{\sijbarsq}{\left( \overline{\bs{s}}^{a}_{ij} \right)^2}
\newcommand{\lijbar}{\overline{\bs{l}}_{ij}}
\newcommand{\clb}{\bs{\underline{c}}_{ij}}
\newcommand{\cub}{\bs{\overline{c}}_{ij}}
\newcommand{\slb}{\bs{\underline{s}}_{ij}}
\newcommand{\sub}{\bs{\overline{s}}_{ij}}

\newcommand{\thetac}{\bs{\theta}^{\bs{c}}_{ij}}
\newcommand{\thetas}{\bs{\theta}^{\bs{s}}_{ij}}

\newcommand{\xc}{x^{\mathbf{c}}}
\newcommand{\xcstar}{\bs{x}^{\mathbf{c}*}}

\newcommand{\zstar}{\mathbf{z}_{ij}^*}
\newcommand{\zc}{y_{\mc{C}}}
\newcommand{\xclb}{\bs{\underline{x}}^\mathbf{c}}
\newcommand{\xcub}{\bs{\overline{x}}^\mathbf{c}}
\newcommand{\betabar}{\bs{\beta}}
\newcommand{\cycconstrs}{lifted cycle constraints}

\usepackage{longtable}
\newcommand{\sm}{\setminus}

\usepackage{mathtools}

\usepackage{multirow}
\usepackage{graphicx}
\usepackage{caption}
\allowdisplaybreaks

\newcommand{\fr}[1]{\mathfrak{#1}}
\newcommand{\bs}[1]{\boldsymbol{#1}}

\newlength{\characterlength}
\usepackage{xr}
\begin{document}

\RUNAUTHOR{Guo et al.}

\RUNTITLE{Tightening QC Relaxations for the ACOTS Problem}

\TITLE{Tightening Quadratic Convex Relaxations for the \\ AC Optimal Transmission Switching Problem}

\ARTICLEAUTHORS{%
\AUTHOR{Cheng Guo}
\AFF{School of Mathematical and Statistical Sciences, Clemson University, Clemson, SC 29634, USA, \EMAIL{cguo2@clemson.edu}}
\AUTHOR{Harsha Nagarajan}
\AFF{Applied Mathematics and Plasma Physics (T-5), Los Alamos National Laboratory, Los Alamos, NM 87545, USA, \EMAIL{harsha@lanl.gov}}
\AUTHOR{Merve Bodur}
\AFF{School of Mathematics and Maxwell Institute for Mathematical Sciences, University of Edinburgh, Edinburgh,\ EH9 3FD, UK, \EMAIL{merve.bodur@ed.ac.uk}}
} 

\ABSTRACT{%
The Alternating Current Optimal Transmission Switching (ACOTS) problem incorporates line switching decisions into the AC Optimal Power Flow (ACOPF) framework, offering well-known benefits in reducing operational costs and enhancing system reliability. ACOTS optimization models contain discrete variables and nonlinear, non-convex constraints, which make it difficult to solve. {\cred In this work, we develop strengthened quadratic convex (QC) relaxations for ACOTS, where we tighten the relaxation with} several new valid inequalities, {\cred including} a novel kind of on/off cycle-based polynomial constraints by taking advantage of the network structure. {\cred We linearize the sum of on/off trilinear terms in the relaxation using extreme-point representation, demonstrating theoretical tightness, and efficiently incorporate on/off cycle-based polynomial constraints through disjunctive programming-based cutting planes. Combined with an optimization-based bound tightening algorithm, this results in the tightest QC-based ACOTS relaxation to date. }{\cblue We additionally propose a novel maximum spanning tree-based heuristic to improve the computational performance by fixing certain lines to be switched on.} {\cblue Our extensive numerical experiments on medium-scale PGLib instances show significant improvements on relaxation bounds, while tests on large-scale instances with up to 2,312 buses demonstrate substantial performance gains. To our knowledge, this is the first ACOTS relaxation-based approach to demonstrate near-optimal switching solutions on realistic large-scale power grid instances.} 
}%


\maketitle

\section{Introduction}
Transmission switching, with its inception in the 1980s \citep{glavitsch1985switching}, has gained considerable attention in both industry and academia in recent years \citep{hedman2008optimal,kocuk2017new}. The optimal transmission switching (OTS) problem studies how to switch on or off certain transmission lines to modify the network topology in the real-time operation of transmission power grids. Solving the OTS problem brings several benefits that the traditional optimal power flow (OPF) problem solution cannot offer, such as reducing the total operational cost, mitigating transmission congestion, clearing contingencies, and improving engineering limits \citep{hedman2011review}. Thus, as the modern transmission control and relay technologies evolve, transmission line
switching has become an important option in power system operators' toolkits. 

Previous literature on OTS has mainly relied on the DC approximation of the power flow model to avoid the mathematical complexity of the non-convex AC power flow equations.
The first formal mathematical model for the OTS problem, proposed by \cite{fisher2008optimal}, relies on this DC approximation. \cite{kocuk2016cycle} developed a cycle-induced relaxation for a DCOTS model, and characterized its convex hull.

The drawback of DCOTS approximation is that the optimal decisions may not accurately represent power flows or may even be infeasible in the AC setting \citep{coffrin2014primal}. This drawback motivates the adoption of ACOTS. The ACOTS problem can be formulated as a non-convex mixed-integer nonlinear program (MINLP), which is challenging to solve. Moreover, even in the DC setting, the OTS problem is known to be NP-hard \citep{lehmann2014complexity}. 
Several heuristics are proposed for solving the ACOTS problem, e.g., by \cite{barrows2014correcting} and \cite{goldis2013applicability}. Separately, convex relaxations for ACOTS have gained significant attention in recent years. 

The convex relaxations of the ACOTS problem, of which the literature is quite scarce, are built upon the rich literature on ACOPF convex relaxations, such as the second-order cone (SOC) relaxation \citep{jabr2006radial}, the quadratic convex (QC) relaxation \citep{coffrin2015qc}, and the semi-definite programming (SDP) relaxation \citep{bai2008semidefinite}. For the ACOPF problem, both standard SDP and QC relaxations are at least as strong as the SOC relaxation, while the SDP and QC relaxation strengths are not directly comparable. Computationally, the SOC and QC relaxations are faster and more reliable than the SDP relaxation \citep{coffrin2015qc}. In the ACOTS setting, \cite{hijazi2017convex} and \cite{bestuzheva2020convex} propose a QC relaxation that incorporates on/off decision variables, which provides a tight lower bound to the generation-cost minimization objective. We further tighten this on/off version of the QC relaxation with a much stronger linearization and additional valid inequalities, with some of the latter being novel. 

In particular, one type of valid inequalities we incorporate is cycle-based polynomial constraints (``lifted cycle constraints" for short). The \cycconstrs \ were first proposed by \cite{kocuk2016strong} as a relaxation to the arc-tangent constraints in ACOPF. Their work provides a reformulated SOC relaxation for ACOPF, strengthened by the McCormick relaxation of the  \cycconstrs, which is shown to be incomparable to the SDP relaxation. Further, these cycle-based valid inequalities were generalized for the ACOTS problem by \cite{kocuk2017new}, which was derived in the rectangular co-ordinates. However, in our work, we develop a new type of \cycconstrs \ based on the QC relaxation of the ACOPF problem, written in polar co-ordinates. Our method  takes advantage of the direct access to auxiliary variables representing the trigonometric functions in the QC relaxation. Then we reformulate the on/off version of those \cycconstrs~for the ACOTS problem, which is new in the literature. In our strengthened formulation, we combine those new constraints and the \cycconstrs \ from \citep{kocuk2017new}. Further, we linearize these {\cred on/off} polynomial constraints with the tightest extreme-point representation, which captures the convex hull of the {\cred on/off} \cycconstrs~for a given cycle. {\cred Since adding all constraints at the same time can make the problem challenging to solve, we derive novel cutting planes and incorporate \cycconstrs \ via a branch-and-cut scheme.}

We further improve the bounds using an optimization-based bound tightening (OBBT) technique. OBBT is often used in MINLPs to tighten relaxation bounds \citep{nagarajan2019adaptive}. It has been shown to be an effective bound-tightening method in nonlinear AC power flow models. \cite{sundar2018optimization} use OBBT to tighten an ACOPF-QC relaxation, whereas \cite{bestuzheva2020convex} use OBBT for an ACOTS-QC relaxation with nonlinear terms linearized using weaker recursive-McCormick relaxations. Additionally, \cite{fattahi2018bound} employs bound tightening methods for the simpler DCOTS model, testing them on large-scale networks. 
In our work, we incorporate all the proposed tightening valid inequalities and the ones from the literature within OBBT, to achieve a very tight lower bound for the ACOTS problem.

Our main contributions can be summarized as follows:
    
    (1) We strengthen the ACOTS-QC relaxation with several techniques. First, we linearize the {\cred summation of two} on/off trilinear terms with the tight extreme-point representation {\cred and prove such a linearization captures its convex hull}. To the best of our knowledge, we are the first to {\cred develop this linearization method for} ACOTS. We also reformulate several ACOPF-QC strengthening constraints for the OTS setting, some of which are novel. Compared with the state-of-the-art on/off QC relaxation formulation \citep{coffrin2018power}, our strengthened relaxation is shown to provide better lower bounds for many test cases, and some of those improvements are substantial. {\cblue It can also solve several large-scale PGLib instances with 500 to 2,312 buses, which, to the best of our knowledge, have not been previously solved by the ACOTS-QC relaxation.}

    (2) We derive a new lifted cycle constraint that strengthens the QC relaxations of both ACOPF and ACOTS. We linearize those constraints with the extreme-point representation, which is always tighter than the recursive-McCormick relaxation in the literature. Also, to separate \cycconstrs \ with discrete decisions, we develop {\cred novel cutting planes by reformulating Benders cuts via disjunctive programming. We incorporate those cutting planes in a branch-and-cut framework, which leads to significant solution time reductions.} {\cblue The \cycconstrs~are also useful for obtaining solutions of large-scale PGLib instances.}

    (3) {\cred We combine strengthening methods listed in (1) and (2) with OBBT and} obtain the tightest {\cred QC-based} ACOTS relaxation in the literature, to the best of our knowledge.

    {\cblue (4) We propose a novel heuristic that solves a \textit{restricted} ACOTS-QC model by fixing a maximum spanning tree of the network to remain switched on. This approach significantly improves computational performance on medium-scale networks, enables the solution of large-scale instances that were previously unsolved, and provides near-optimal solutions.}

\section{The ACOTS Problem}\label{sec: setup}
{\cred Before describing the model for the ACOTS problem, we first introduce some notation.} Throughout, constants are typeset in boldface to make it easier to distinguish between decision variables and parameters. In the AC power flow equations, upper case letters represent complex quantities. $\fr{R}(\cdot)$ and $\fr{T}(\cdot)$ respectively denote the real and imaginary parts of a complex number. Given any two complex numbers (variables/constants) $z_1$ and $z_2$, $z_1 \geqslant z_2$ implies $\fr{R}({z}_1) \geqslant \fr{R}({z}_2)$ and $\fr{T}({z}_1) \geqslant \fr{T}({z}_2)$. $\mid\cdot\mid$ and $(\cdot)^*$ represent the magnitude and Hermitian conjugate of a complex number, respectively. When applied on a real-valued number, $\mid\cdot\mid$ represents its absolute value. 
{\cred The notation for sets, parameter and variables is summarized in Table \ref{table: notation}.}

{
\small
\begin{longtable}[htbp]{ll}
\caption{Notation}\label{table: notation}\\
\toprule
\multicolumn{2}{l}{\textbf{Sets and parameters:}}       \\*[0.1cm]
$\mc{N}$ & Set of buses (nodes).\\
$\mc{N}^L$ & Set of leaf buses with no loads.\\
$\mc{N}^{\text{ref}}$ & Set of reference buses.\\
$\cal{G}$ & Set of generators.\\
${\cal{G}}_i$ & Set of generators at bus $i$.\\
$\mc{A}$ & Set of lines (arcs).\\
$\mc{A}^{R}$ & Set of arcs in reversed direction.\\
\cblue $\bs{c}^{\bs{g}}_0, \bs{c}^{\bs{g}}_1, \bs{c}^{\bs{g}}_2$ & Generation cost coefficients.\\
$\img$ & Unit imaginary number.\\
$\Yij = \gij +\img\bs{b}_{ij}$ & Admittance on line $(i,j)$.\\
$\Ycij = \gcij + \img \bcij$ & Charging admittance on line $(i,j)$.\\
$\Ysij = \gsij + \img \bsij$ & Shunt admittance on line $(i,j)$.\\
$\Tij = \tRij + \img \tIij$ & Branch complex transformation ratio (tap ratio) on line $(i,j)$.\\
$\bs{S}_{i}^{\bs{d}} = \bs{p}_{i}^{\bs{d}} + \img\bs{q}_{i}^{\bs{d}}$ & AC power demand at bus $i$.\\
\cblue $\clb, \cub$ & \cblue Cosine function ($\cos({\theta}_{ij})$) bounds on line $(i,j)$.\\
\cblue $\slb, \sub$ & \cblue Sine function ($\sin({\theta}_{ij})$) bounds on line $(i,j)$.\\
{\cblue $\sijbar$} & Apparent power bound on line $(i,j)$.\\
$\thetalb, \thetaub$ & Phase angle difference bounds on line $(i,j)$.\\
$\underline{\bs{v}}_i, \overline{\bs{v}}_i$ & Voltage magnitude bounds at bus $i$.\\
$ \underline{\bs{S}}_{i}^{\bs{g}}$, $\overline{\bs{S}}_{i}^{\bs{g}}$ & Power generation bounds at bus $i$.\\
$\lijbar$ & Current magnitude squared upper limit on line $(i,j)$.\\
\midrule
 \multicolumn{2}{l}{\textbf{Variables:}} \\*[0cm]
$v_i$ & Voltage magnitude at bus $i$.\\
$\theta_i$ & Voltage angle at bus $i$.\\
${V_i} = v_i\bs{e}^{\img\theta_i}$ & AC complex voltage at bus $i$.\\
$\theta_{ij}$ & Phase angle difference on line $(i,j)$.\\
$w_i$ & Squared voltage magnitude at bus $i$.\\
$W_{ij} = w^R_{ij} + \img w^I_{ij}$ & AC voltage product on line $(i,j)$.\\
$S_{ij} = p_{ij} +\img q_{ij}$ & AC power flow on line $(i,j)$.\\
$S^g_k= p^g_k+\img q^g_k$ & AC power generation of generator $k$.\\
$l_{ij}$ & Current magnitude squared on line $(i,j)$.\\
$z_{ij}$ & 1 if line $(i,j)$ is switched on, 0 otherwise.\\ 
\bottomrule
\end{longtable}
}

The power network can be represented with the graph $G = (\mc{N}, \mc{A})$, where $\mc{N}$ corresponds to the set of buses, while the set of arcs $\mc{A}$ corresponds to the set of lines. Note that we assume lines in the network are directed with designated from/to buses, as indicated by the data. This assumption is conventional, and it is necessary because the data contains asymmetric shunt conductance and transformers.

The ACOTS problem minimizes the total production cost of generators such that all the demands at the buses, the physical constraints (e.g., Ohm's and Kirchoff's law), and engineering limit constraints (e.g., transmission line flow limits) are satisfied. In this work, for the purpose of the QC relaxation (see Section \ref{sec: qc}), we model the AC power flow equations in \textit{polar} co-ordinates \citep{taylor2015convex}, thus the following ACOTS formulation:
\begin{subequations}
\label{eq:ACOTS}
\begin{align}
\label{eq: ots_objective} 
\min&\sum_{k\in \cal{G}} \left(\bs{c}^{\bs{g}}_{2k}(p_k^{g})^2 + \bs{c}^{\bs{g}}_{1k} p_k^g \right) + \sum_{i\in\mc{N}\sm\mc{N}^L}\sum_{k\in\mc{G}_i} \bs{c}^{\bs{g}}_{0k} + \sum_{i\in\mc{N}^L}\sum_{k\in\mc{G}_i}\bs{c}^{\bs{g}}_{0k} z_{\{ft | (f,t) \in\mc{A}, f=i \text{~or~} t = i\}} \hspace{-6cm}\\
\label{eq: ots_s_balance}\st&\sum_{k \in {\cal{G}}_i} S_k^g - \bs{S}_{i}^{\bs{d}} - \Yij^{s*} w_i = \sum_{(i,j)\in\mc{A} \cup \mc{A}^R}S_{ij} && \forall i \in \cal{N} \\
&\label{eq: ots_sij}S_{ij} = (\Yij + \Ycij)^* \frac{w_i}{|\Tij|^2}z_{ij} - \Yij^* \frac{W_{ij}}{\Tij} && \forall (i,j) \in \mc{A} \\
&\label{eq: ots_sji}S_{ji} = (\Yij + \Ycji)^* w_j z_{ij} - \Yij^* \frac{W^*_{ij}}{\bs{T}^*_{ij}}\hspace{-3mm}&&\forall (i,j) \in \mc{A}\\
&\label{eq: ots_wii} w_i = v_i^2 &&  \forall i \in \cal{N}\\
&\label{eq: ots_wij} W_{ij} = V_iV_j^* z_{ij} &&  \forall (i,j) \in \mc{A}\\
& \label{eq: ots_link}   \theta_{ij} = \theta_i - \theta_j && \forall (i,j) \in \mc{A}\\
& \label{eq: ots_theta} \underline{\bs{\theta}}_{ij}z_{ij} - \thetaM (1 - z_{ij}) \leqslant \theta_{ij} \leqslant \overline{\bs{\theta}}_{ij} z_{ij} + \thetaM (1 - z_{ij}) && \forall (i,j) \in \mc{A}\\
& \label{eq: ots_refbus} \theta_i = 0 && \forall i\in\mc{N}^{\text{ref}}\\
&\label{eq: ots_g_cap} \underline{\bs{S}}_{i}^{\bs{g}} \leqslant S^g_k \leqslant \overline{\bs{S}}_{k}^{\bs{g}} &&  \forall k \in \cal{G} \\
&\label{eq: ots_s_cap}|S_{ij}|^2 \leqslant {\cblue \sijbarsq} z_{ij}^2, \quad |S_{ji}|^2 \leqslant {\cblue \sijbarsq} z_{ij}^2 && \forall (i,j) \in \mc{A}\\
&\label{eq: ots_v_limit} \vlb_i \leqslant v_i \leqslant \vub_i && \forall i\in\mc{N}\\
&\label{eq: z_var} z_{ij} \in \{0, 1\} && \forall (i, j)\in\mc{A},
\end{align}
\end{subequations}
where $z_{\{ft | (f,t) \in\mc{A}, f=i \text{~or~} t = i\}}$ is the switch on/off variable for a line having either end connected to a leaf node with 0 load. When such a line is switched off, the generators on the leaf node are disconnected from the network, and thus we do not need to pay the fixed cost $c_{0k}$. 

The convex quadratic objective \eqref{eq: ots_objective} minimizes total generator dispatch cost. Constraints \eqref{eq: ots_s_balance} correspond to the power balance at each bus, i.e., Kirchoff's current law. Constraints \eqref{eq: ots_sij} to \eqref{eq: ots_wij} model the power flow on each line. Note that constraints \eqref{eq: ots_sij} and \eqref{eq: ots_sji} ensure that the power flow over line $(i, j)$ is zero if the line is switched off; constraints \eqref{eq: ots_wij} ensure that $W_{ij} = 0$ when line $(i, j)$ is switched off. Constraints \eqref{eq: ots_link} connect voltage angle and voltage difference variables. 

Constraints \eqref{eq: ots_theta} limit the phase angle difference on each line. We define $\thetau = \max(|\thetalb|, |\thetaub|)$. Let $\bs{\theta}^{\bs{u}, \max}_{ij,k}$ be the $k$th largest value in $\{\thetau~|~(i, j) \in\mc{A}\}$, and $\thetaM = \sum_{k=1}^{|\mc{N}|-1} \bs{\theta}^{\bs{u}, \max}_{ij,k}$ be a big-M constant for phase angle difference. This big-M constant enables us to provide proper bounds for $\theta_{ij}$ in constraints \eqref{eq: ots_theta}. 

Constraints \eqref{eq: ots_refbus} set the voltage angles of reference buses to 0.  Constraints \eqref{eq: ots_g_cap} restrict the apparent power output of each generator. Constraints \eqref{eq: ots_s_cap} are thermal limit constraints that restrict the total electric power transmitted on each line. Note that in these constraints a squared form of $z_{ij}$ is used instead of the linear form, as it produces a tighter formulation \citep{hijazi2017convex}. Constraints \eqref{eq: ots_v_limit} limit the voltage magnitude at each bus.


This ACOTS model is a non-convex MINLP, which contains nonlinear constraints \eqref{eq: ots_sij} through \eqref{eq: ots_wij}.
An easy method to linearize constraints \eqref{eq: ots_sij} and \eqref{eq: ots_sji} is to introduce a lifted variable $w^z_{ij}$ per line, which equals $w_i$ when $z_{ij} = 1$ and 0 otherwise, then constraints \eqref{eq: ots_sij} and \eqref{eq: ots_sji} can be replaced with the following linear constraints for every line $(i,j) \in \mc{A}$:
\begin{subequations}
\label{eq:linearize_sij}
\begin{alignat}{3}
&\label{eq: ots_sij_lift}S_{ij} = (\Yij + \Ycij)^* \frac{w^z_{ij}}{|\Tij|^2} - \Yij^* \frac{W_{ij}}{\Tij}  \\
&\label{eq: ots_sji_lift}S_{ji} = (\Yij + \Ycji)^* w^z_{ji} - \Yij^* \frac{W^*_{ij}}{\bs{T}^*_{ij}}\\
& \label{eq: ots_wz_bound1} w_i - (1-z_{ij}) \vub_i^2 \leqslant w^z_{ij} \leqslant w_i - (1 - z_{ij}) \vlb_i^2  \\
& \label{eq: ots_wz_bound2} w_j - (1-z_{ij}) \vub_j^2 \leqslant w^z_{ji} \leqslant w_j - (1 - z_{ij}) \vlb_j^2 \\
&  \label{eq: ots_wz_bound3}\vlb_i^2 z_{ij} \leqslant w^z_{ij} \leqslant \vub_i^2 z_{ij} \\
&  \label{eq: ots_wz_bound4}\vlb_j^2 z_{ij} \leqslant w^z_{ji} \leqslant \vub_j^2 z_{ij}.
\end{alignat}
\end{subequations}
Constraints \eqref{eq: ots_sij_lift} - \eqref{eq: ots_wz_bound2} are from \citep{hijazi2017convex}, while constraints \eqref{eq: ots_wz_bound3} and \eqref{eq: ots_wz_bound4} are new and provide tighter bounds for $w^z_{ij}$ and $w^z_{ji}$. Tractable convex relaxations for constraints \eqref{eq: ots_wii} and \eqref{eq: ots_wij} are described in the next section.
\section{On/Off Quadratic Convex Relaxation}\label{sec: qc}
The on/off version of the QC relaxation for the ACOTS model relaxes nonlinear constraints \eqref{eq: ots_wii} and \eqref{eq: ots_wij}. 
Let $w^R_{ij} := \fr{R}(W_{ij})$ and $w^I_{ij} := \fr{T}(W_{ij})$, then \eqref{eq: ots_wij} can be equivalently written as follows $(\forall (i, j) \in\mc{A})$: 
\begin{subequations}\label{eq:w_nonconv}
\begin{align}
&\label{eq:wcs} w^R_{ij} = z_{ij}\left(v_iv_j\cos(\theta_{ij})\right)  \\
& \label{eq:wsn} w^I_{ij} = z_{ij}\left(v_iv_j\sin(\theta_{ij})\right). 
\end{align}
\end{subequations}

A key feature of the QC relaxation is the use of polar co-ordinates, which has direct access to voltage magnitude $v_i$ and voltage angle $\theta_i$ variables. 
This enables stronger links between voltage variables. In what follows, we list the QC relaxation constraints, some of which are based on previous works \citep{coffrin2015qc,hijazi2017convex}, 
while others are newly derived, which we will specify.

(1) \textit{Quadratic function relaxation ($\forall i\in\mc{N}$): }We can formulate the convex-hull envelope for 
constraints \eqref{eq: ots_wii} by relaxing every equality to a quadratic inequality constraint \eqref{eq:MC-q1}, {\cblue as done in \citep{hijazi2017convex}, } 
and providing upper bounds via linear McCormick relaxation constraints in \eqref{eq:MC-q2}: 
\begin{subequations}\label{eq:MC-q}
\begin{align}
& {w}_i \geqslant v_i^2\label{eq:MC-q1}\\
& {w}_i \leqslant (\vlb_i + \vub_i)v_i - \vlb_i \vub_i\label{eq:MC-q2} 
\end{align}
\end{subequations}

(2) \textit{Cosine and sine function relaxations $(\forall (i, j) \in \mc{A})$: }
The trigonometric term $\cos(\theta_{ij})$ in \eqref{eq:wcs} is non-convex. 
We define a lifted variable $c_{ij}$ that captures the the convex envelope of 
$\cos(\theta_{ij})$. 
We assume $\thetau \leqslant \pi/2$, which is reasonable as the absolute value of the phase angle 
differences across the lines is usually under 15 degrees in practice \citep{purchala2005usefulness}. We define
the following constants for every line $(i,j) \in \mc{A}$ which are necessary for the relaxation constraints: 
\begin{align*}
    \thetac = \frac{\cos(\thetaub) - \cos(\thetalb)}{\thetaub-\thetalb}, \quad 
    \thetas = \frac{\sin(\thetaub) - \sin(\thetalb)}{\thetaub-\thetalb}
\end{align*}
Following the disjunctive programming method in \citep{hijazi2017convex}, an on/off version of the
convex relaxation of the cosine function is as follows:
\begin{subequations} \label{eq: ots_cosine}
\begin{align}
& - c_{ij} + \thetac \theta_{ij} \leqslant \Big(\thetac \thetalb - \cos(\thetalb)\Big) z_{ij}  + \left|\thetac \right| \thetaM (1 - z_{ij})\label{eq: ots_cosine_sine_1} \\
& c_{ij} \leqslant z_{ij} - \frac{1- \cos(\thetau)}{(\thetau)^2} \theta_{ij}^2 + \frac{1- \cos(\thetau)}{(\thetau)^2} (\thetaM)^2 (1 - z_{ij})\label{eq: ots_cosine_sine_2}\\
& \clb z_{ij} \leqslant c_{ij} \leqslant \cub z_{ij}. \label{eq: ots_cosine_sine_7}
\end{align}
\end{subequations}%
Constraints \eqref{eq: ots_cosine_sine_1} and \eqref{eq: ots_cosine_sine_2} are big-M constraints. 
When $z_{ij} = 1$, they represent quadratic convex relaxations of $cos(\theta_{ij})$ 
derived from trigonometric identities and properties of quadratic functions, and those convex 
relaxations are not valid when the line $(i, j)$ is switched off, as we have $c_{ij} = 0$. 
Therefore, big-M parameter $\thetaM$ is used to ensure that when $z_{ij} = 0$, those 
constraints are valid for $c_{ij} = 0$ and $\theta_{ij}\in[\thetalb, \thetaub]$. 
Constraint \eqref{eq: ots_cosine_sine_7} provides bounds for $c_{ij}$, and ensure that $c_{ij} = 0$ 
when the line $(i,j)$ is switched off. Note that constraint \eqref{eq: ots_cosine_sine_1} is a 
new constraint that is not in \citep{hijazi2017convex} or \citep{bestuzheva2020convex}.

Similarly, we define a lifted variable $s_{ij}$ that captures the the convex envelope of 
$\sin(\theta_{ij})$. 
When $\thetau \leqslant \pi/2$, 
a disjunctive relaxation of the sine function is as follows: 
\begin{subequations}\label{eq: acots_sine}
\begin{align}
& s_{ij} - \cos\left(\frac{\thetau}{2}\right) \theta_{ij} \leqslant \left(\sin\left(\frac{\thetau}{2}\right) - \cos\left(\frac{\thetau}{2}\right) \frac{\thetau}{2}\right) z_{ij} + \cos\left(\frac{\thetau}{2}\right)\thetaM (1-z_{ij}),~\text{if~} \thetaub \geqslant 0\label{eq: ots_cosine_sine_3}\\
- & s_{ij} + \cos\left(\frac{\thetau}{2}\right) \theta_{ij} \leqslant \left(\sin\left(\frac{\thetau}{2}\right)- \cos\left(\frac{\thetau}{2}\right) \frac{\thetau}{2} \right) z_{ij} + \cos\left(\frac{\thetau}{2}\right)\thetaM (1-z_{ij}),~\text{if~}\thetalb \leqslant 0\\
& s_{ij} - \thetas \theta_{ij} \leqslant \left( - \thetas \thetalb + \sin(\thetalb)\right) z_{ij} + \thetas \thetaM (1 - z_{ij}),~\text{if~} \thetaub \leqslant 0\\
-& s_{ij} + \thetas \theta_{ij} \leqslant \left( \thetas \thetalb  - \sin(\thetalb)\right) z_{ij} + \thetas \thetaM (1 - z_{ij}),~\text{if~}\thetalb \geqslant 0\label{eq: ots_cosine_sine_6}\\
& \slb z_{ij} \leqslant s_{ij} \leqslant \sub z_{ij}\label{eq: ots_cosine_sine_8},
\end{align}
\end{subequations}
where constraints \eqref{eq: ots_cosine_sine_3} - \eqref{eq: ots_cosine_sine_6} are derived from 
linear outer approximation of the $\sin(\theta_{ij})$ function \citep{hijazi2017convex}. 
 
(3) \textit{Extreme-point representation for summation of on/off trilinear terms $(\forall (i, j) \in \mc{A})$: }{\cred Next, we develop a novel extreme-point representation to linearize the sum of two on/off trilinear terms and theoretically prove its tightness. We first} substitute the non-convex functions $\cos(\theta_{ij})$ and $\sin(\theta_{ij})$ in constraints \eqref{eq:w_nonconv}
with lifted variables $c_{ij}$ and $s_{ij}$ and their convex envelopes, respectively. The remaining non-linearities reduce to trilinear terms $v_i v_j c_{ij}$ and $v_i v_j s_{ij}$, which are further controlled by the status of the on/off variable $z_{ij}$. To linearize these terms, we generalize the convex hull representation of the extreme-point formulation \citep{lu2018tight} to incorporate on/off variables, as discussed below.

Let the extreme points of the domain $[\underline{\bs{v}}_i, \overline{\bs{v}}_i] \times [\underline{\bs{v}}_j, \overline{\bs{v}}_j]  \times [\underline{\bs{c}}_{ij}, \overline{\bs{c}}_{ij}]$ be denoted by $\bs{\xi}^k$ with $k = 1,\hdots,8$, and the extreme points
of the domain $[\underline{\bs{v}}_i, \overline{\bs{v}}_i] \times [\underline{\bs{v}}_j, \overline{\bs{v}}_j]  \times [\underline{\bs{s}}_{ij}, \sub]$ be denoted by $\bs{\gamma}^k, k = 1,\hdots,8$. We relax constraints \eqref{eq:w_nonconv} as follows:
\begin{subequations} 
\label{eq:triform}
\allowdisplaybreaks
\begin{align}
& w^R_{ij} = \sum_{k=1}^8 \lambda^c_{ij, k} \,({\bs{\xi}_1^k} {\bs{\xi}_2^k} {\bs{\xi}_3^k}), \quad w^I_{ij} = \sum_{k=1}^8 \lambda^s_{ij, k} \,({\bs{\gamma}_1^k} {\bs{\gamma}_2^k} {\bs{\gamma}_3^k}) \label{eq:triform1}\\
&\sum_{k=1}^8 \lambda^c_{ij, k} {\bs{\xi}_1^k} + (1 - z_{ij})\vlb_i \leqslant v_i \leqslant \sum_{k=1}^8 \lambda^c_{ij, k} {\bs{\xi}_1^k} + (1 - z_{ij})\vub_i
\label{eq:triform_2}\\
&\sum_{k=1}^8 \lambda^s_{ij, k} {\bs{\gamma}_1^k} + (1 - z_{ij})\vlb_i \leqslant v_i \leqslant \sum_{k=1}^8 \lambda^s_{ij, k} {\bs{\gamma}_1^k} + (1 - z_{ij})\vub_i\label{eq:triform_3}\\ 
& \sum_{k=1}^8 \lambda^c_{ij, k} {\bs{\xi}_2^k} + (1 - z_{ij})\vlb_j \leqslant v_j \leqslant \sum_{k=1}^8 \lambda^c_{ij, k} {\bs{\xi}_2^k} + (1 - z_{ij})\vub_j \label{eq:triform_4}\\
&\sum_{k=1}^8 \lambda^s_{ij, k} {\bs{\gamma}_2^k} + (1 - z_{ij})\vlb_j \leqslant v_j \leqslant \sum_{k=1}^8 \lambda^s_{ij, k} {\bs{\gamma}_2^k} + (1 - z_{ij})\vub_j \label{eq:triform_5}\\
& c_{ij} = \sum_{k=1}^8 \lambda^c_{ij, k} \, {\bs{\xi}_3^k}, \  s_{ij} = \sum_{k=1}^8 \lambda^s_{ij, k} \, {\bs{\gamma}_3^k} \label{eq:triform_6}\\
& \sum_{k=1}^8 \lambda^c_{ij, k} = z_{ij}, \;\; \sum_{k=1}^8 \lambda^s_{ij, k} = z_{ij}, \;\; \lambda^c_{ij, k} \geqslant 0, \  \lambda^s_{ij, k} \geqslant 0,~\forall k=1,\dots,8  \label{eq:triform_7}\\
 &  \begin{bmatrix}
        \lambda^c_{ij,1} + \lambda^c_{ij,2} - \lambda^s_{ij,1} - \lambda^s_{ij,2}\\
        \lambda^c_{ij,3} + \lambda^c_{ij,4} - \lambda^s_{ij,3} - \lambda^s_{ij,4} \\
        \lambda^c_{ij,5} + \lambda^c_{ij,6} - \lambda^s_{ij,5} - \lambda^s_{ij,6} \\
        \lambda^c_{ij,7} + \lambda^c_{ij,8} - \lambda^s_{ij,7} - \lambda^s_{ij,8} 
    \end{bmatrix}^\top 
    \begin{bmatrix}
    \underline{\bs{v}}_i \cdot \underline{\bs{v}}_j \\
    \underline{\bs{v}}_i \cdot \overline{\bs{v}}_j \\
    \overline{\bs{v}}_i \cdot \underline{\bs{v}}_j \\
    \overline{\bs{v}}_i \cdot \overline{\bs{v}}_j 
    \end{bmatrix} = 0, \label{eq:triform_link}
\end{align}
\end{subequations}
where $\bs{\xi}_1^k$ is the value of $v_i$ in the extreme point $\bs{\xi}^k$. The constants $\bs{\xi}_2^k$, $\bs{\xi}_3^k$, $\bs{\gamma}_1^k$, $\bs{\gamma}_2^k$, and $\bs{\gamma}_3^k$ are similarly defined. $\lambda_{ij,k}^c$ and $\lambda_{ij,k}^s$ are auxiliary multiplier variables for representing a linear combination of the extreme points. When $z_{ij} = 1$, constraints \eqref{eq:triform1} connect values of $w^R_{ij}$ and $w^I_{ij}$ with convex combinations of extreme points for trilinear terms; constraints \eqref{eq:triform_2} - \eqref{eq:triform_6} equate the values of $v_i$, $v_j$, $c_{ij}$, and $s_{ij}$ to convex combinations of their respective extreme points in $[\underline{\bs{v}}_i, \overline{\bs{v}}_i] \times [\underline{\bs{v}}_j, \overline{\bs{v}}_j]  \times [\underline{\bs{c}}_{ij}, \overline{\bs{c}}_{ij}]$. When $z_{ij} = 0$, constraints \eqref{eq:triform1} - \eqref{eq:triform_6} enforce $w^R_{ij} = w^I_{ij} = c_{ij} = s_{ij} = 0$, and impose no constraints on $v_i$ and $v_j$. Constraints \eqref{eq:triform_7} ensure that the summations of convex combination coefficients equal to 1 when line $(i, j)$ is switched on, and all the coefficients become 0 when the line is switched off. Linking constraints \eqref{eq:triform_link} connect the shared bilinear term $v_i v_j$ that appears in both trilinear terms $v_i v_j \cos(\theta_{ij})$ and $v_i v_j \sin(\theta_{ij})$.

Note that when $z_{ij} = 1$, constraints \eqref{eq:triform_2} - \eqref{eq:triform_5} and \eqref{eq:triform_7} reduce to the following constraints:
\begin{subequations}\label{eq: triform_v_on} 
\begin{align}
    & v_i = \sum_{k=1}^8 \lambda^c_{ij, k} {\bs{\xi}_1^k}  = \sum_{k=1}^8 \lambda^s_{ij, k} {\bs{\gamma}_1^k},\label{eq: triform_v_on1}\\
    & v_j = \sum_{k=1}^8 \lambda^c_{ij, k} {\bs{\xi}_2^k} = \sum_{k=1}^8 \lambda^s_{ij, k} {\bs{\gamma}_2^k}, \label{eq: triform_v_on2} \\ 
    & \sum_{k=1}^8 \lambda^c_{ij, k} = 1, \;\; \sum_{k=1}^8 \lambda^s_{ij, k} = 1. \label{eq: triform_lambda_on}
\end{align}
\end{subequations} 

We next formally show that the constraints in \eqref{eq:triform} form the \textit{tightest, or the convex hull}, relaxation for the summation of nonlinear terms of the form 
\begin{align}
z_{ij}\left( {\bs a_1} v_i v_j c_{ij} + {\bs a_2} v_i v_j s_{ij}\right) \quad \forall (i,j) \in \mc{A} \cup \mc{A}^R.    
\label{eq:summation}
\end{align}
Note that this form appears in constraints \eqref{eq: ots_sij} and \eqref{eq: ots_sji}, where ${\bs a_1}$ and ${\bs a_2}$ represent coefficients that are functions of $\gij$, $\bs{b}_{ij}$, $\bs{t}^R_{ij}$, and $\bs{t}^I_{ij}$. For this purpose, we define the following: Let $\eta = $ $(w^R_{ij},$ $ w^I_{ij},$ $c_{ij},$ $s_{ij},$ $\lambda^c_{ij,1},$ $\ldots,$ $\lambda^c_{ij,8},$ $\lambda^s_{ij, 1},$ $\ldots,$ $\lambda^s_{ij,8},$ $z_{ij},$ $v_i,$ $v_j)$ and define the set $H = \{\eta| \eta~\text{satisfies}~\eqref{eq:triform}, z_{ij} \in [0, 1]\}$. When $z_{ij} = 0$ and $z_{ij} = 1$, $H$ becomes $H^0$ and $H^1$, respectively:
\begin{align*}
  & H^0 = \Set{ \eta \ | \begin{array}{l}
    w^R_{ij} = w^I_{ij} = c_{ij} = s_{ij} = z_{ij} = 0 \\
    \lambda^c_{ij,k} = \lambda^s_{ij, k} = 0, \forall k=1,\cdots, 8\\
    v_i\in [\vlb_i, \vub_i], v_j\in [\vlb_j, \vub_j]
  \end{array}}\\
  & H^1 = \{\eta \ | \eta~\text{satisfies}~\eqref{eq:triform1}, \eqref{eq:triform_6}, \eqref{eq:triform_link} ~\text{and}~\eqref{eq: triform_v_on}\}
\end{align*}

\cite{sundar2018optimization} show, for a simpler case when $z_{ij} = 1$, that the linearization defined by $H^1$ is the convex hull of the summation terms in \eqref{eq:summation} due to the addition of equality constraints \eqref{eq:triform_link}. However, to understand the tightness (convex hull property) of the linearization \eqref{eq:triform} for the generalized ACOTS model, we present the following theorem, based on the literature of perspective formulations for disjunctive programming \citep{ceria1999convex,nagarajan2019convex}:

\begin{theorem}\label{th: convex_hull}
    $H = \conv(H^0 \cup H^1)$.
\end{theorem}
\proof{Proof.}
First, we prove that $\conv(H^0 \cup H^1) \subseteq H$. For any $\eta^0 \in H^0$, $\eta^0$ satisfies constraints in \eqref{eq:triform} when $z_{ij} = 0$. Similarly, for any $\eta^1 \in H^1$, $\eta^1$ satisfies constraints in \eqref{eq:triform} when $z_{ij} = 1$. Thus, $H^0 \cup H^1 \subseteq H$. Since $H$ contains only linear constraints and is thus convex, we have $\conv(H^0 \cup H^1) \subseteq H$.
    
    Next, we prove that $H \subseteq \conv(H^0 \cup H^1)$. Let $\eta^*\in H$. If $z^*_{ij} = 0$, then $\eta^*\in H^0$, and if $z^*_{ij} = 1$, then $\eta^*\in H^1$. When $z^*_{ij} \in (0,1)$, we define the following variables:
        \begin{align*}
            & \eta_0^* = \left(0, 0, \cdots, 0, \frac{v_i^* - \sum_{k=1}^8 \lambda^{c*}_{ij, k}\bs{\xi}_1^k}{1 - z_{ij}^*},\frac{v_j^* - \sum_{k=1}^8 \lambda^{c*}_{ij, k}\bs{\xi}_2^k}{1 - z_{ij}^*}\right)\\
            & \eta_1^* = \Bigg(\frac{w^{R*}_{ij}}{z^*_{ij}}, \frac{w^{I*}_{ij}}{z^*_{ij}}, \frac{c^*_{ij}}{z^*_{ij}}, \frac{s^*_{ij}}{z^*_{ij}}, \frac{\lambda^{c*}_{ij,1}}{z^*_{ij}}, \cdots, \frac{\lambda^{c*}_{ij,8}}{z^*_{ij}}, \frac{\lambda^{s*}_{ij, 1}}{z^*_{ij}},\cdots,\frac{\lambda^{s*}_{ij,8}}{z^*_{ij}}, 1, \sum_{k = 1}^8 \frac{\lambda^{c*}_{ij, k}}{z^*_{ij}} \bs{\xi}^k_1, \sum_{k = 1}^8 \frac{\lambda^{c*}_{ij, k}}{z^*_{ij}} \bs{\xi}^k_2\Bigg)
        \end{align*}
    Next, we prove that $\eta_0^* \in H^0$ and $\eta_1^* \in H^1$. 
    Because of \eqref{eq:triform_2}, we have $(1-z^*_{ij})\vlb_i \leqslant v_j^* - \sum_{k=1}^8 \lambda^{c*}_{ij, k}\bs{\xi}_2^k\leqslant (1-z^*_{ij})\vub_i$. Thus, $\frac{v_i^* - \sum_{k=1}^8 \lambda^{c*}_{ij, k}\bs{\xi}_1^k}{1 - z_{ij}^*} \in [\vlb_i, \vub_i]$. Similarly, $\frac{v_j^* - \sum_{k=1}^8 \lambda^{c*}_{ij, k}\bs{\xi}_2^k}{1 - z_{ij}^*} \in [\vlb_j, \vub_j]$. Therefore, $\eta_0^* \in H^0$.
    
    For $\eta_1^*$, $\frac{w^{R*}_{ij}}{z^*_{ij}} = \sum_{k=1}^8 \frac{\lambda^{c*}_{ij,1}}{z^*_{ij}} (\bs{\xi}_1^k, \bs{\xi}_2^k, \bs{\xi}_3^k)$ because $ w^{R*}_{ij} = \sum_{k=1}^8 \lambda^{c*}_{ij,1}(\bs{\xi}_1^k, \bs{\xi}_2^k, \bs{\xi}_3^k)$ and $z^*_{ij} \in (0,1)$. Similarly, it can be proved that $\eta_1^*$ satisfies constraints \eqref{eq:triform1}, \eqref{eq:triform_6}, \eqref{eq:triform_link}, and \eqref{eq: triform_lambda_on}. For constraint \eqref{eq: triform_v_on1}, the first equation follows directly from the definition of $\eta_1^*$, and the second equality is correct because the validity of \eqref{eq:triform_2} and \eqref{eq:triform_3} for $\eta^*$ indicates that $\sum_{k=1}^8\lambda^{s*}_{ij,k} \bs{\gamma}^k_1 + (1-z^*_{ij}) \vlb_i \leqslant \sum_{k=1}^8\lambda^{c*}_{ij,k} \bs{\xi}^k_1 + (1-z^*_{ij}) \vlb_i$ and $\sum_{k=1}^8\lambda^{s*}_{ij,k} \bs{\gamma}^k_1 + (1-z^*_{ij}) \vub_i \geq \sum_{k=1}^8\lambda^{c*}_{ij,k} \bs{\xi}^k_1 + (1-z^*_{ij}) \vub_i$, thus $\sum_{k=1}^8\lambda^{c*}_{ij,k} \bs{\xi}^k_1 = \sum_{k=1}^8\lambda^{s*}_{ij,k} \bs{\gamma}^k_1 \Rightarrow \sum_{k=1}^8\frac{\lambda^{c*}_{ij,k}}{z^*_{ij}} \bs{\xi}^k_1 = \sum_{k=1}^8 \frac{\lambda^{s*}_{ij,k}}{z^*_{ij}} \bs{\gamma}^k_1$, which means $\eta_1^*$ satisfies the second equality in \eqref{eq: triform_v_on1}. With similar arguments, $\eta_1^*$ is also feasible for \eqref{eq: triform_v_on2}. Therefore, $\eta_1^* \in H^1$.
    
    Now note that $\eta^* = (1 - z^*_{ij})\eta_0^* + z^*_{ij} \eta_1^*$, which means  $H\subseteq \conv(H^0 \cup H^1)$.\Halmos
    \endproof

To the best of our knowledge, this is the first  attempt at applying the convex hull-based extreme-point formulation for the ACOTS QC relaxation. Previous works \citep{hijazi2017convex,lu2017optimal,bestuzheva2020convex} have utilized recursive McCormick-based relaxations which are not as tight as the above formulation, as the former relaxation when applied to \eqref{eq:summation} is not as tight as $H^1$.

(4) {\it Other valid constraints for strengthening the on/off QC relaxation $(\forall (i,j)\in\mc{A})$: }We also add the following constraints to strengthen the on/off QC relaxation:
\begin{subequations} 
\label{eq: qc_constr}
\begin{align}
&\label{eq: pad} \tan(\thetalb) w^R_{ij} \leqslant w^I_{ij} \leqslant \tan(\thetaub) w^R_{ij} \\
& \label{eq: ots_lnc1} \vsigma_i \vsigma_j (\cos(\phiij) w^R_{ij}  + \sin(\phiij) w^I_{ij}) - \vub_j \cos(\deltaij)\vsigma_i w^z_{ij} \nonumber\\ 
    &\hspace{1mm} - \vub_i \cos(\deltaij) \vsigma_i w^z_{ji} \geq \vub_i \vub_j \cos(\deltaij)(\vlb_i \vlb_j - \vub_i \vub_j) z_{ij} \\
& \label{eq: ots_lnc2} \vsigma_i \vsigma_j (\cos(\phiij) w^R_{ij}  + \sin(\phiij) w^I_{ij}) - \vlb_j \cos(\deltaij)\vsigma_i w^z_{ij}\nonumber\\&\hspace{1mm} - \vlb_i \cos(\deltaij) \vsigma_i w^z_{ji} \geq \vlb_i \vlb_j \cos(\deltaij)(\vub_i \vub_j - \vlb_i \vlb_j) z_{ij} \\
& \label{eq: l_strenghten} |S_{ij}|^2 \leqslant \frac{w_i}{|\Tij|^2}l_{ij}  \\
& \label{eq: ots_rsoc} l_{ij} = |\Yij|^2\left(\frac{w^z_{ij}}{|\Tij|^2} + w^z_{ji} - 2(\tRij w^R_{ij} + \tIij w^I_{ij})/|\Tij|^2\right)\nonumber\\&\hspace{1cm} - \frac{|\Ycij|^2}{|\Tij|^2} w^z_{ij} + 2(\gcij p_{ij} - \bcij q_{ij})\\
& \label{eq: l_bd} 0 \leqslant l_{ij}  \leqslant \lijbar {\cblue z_{ij}},
\end{align}
\end{subequations}
where $\vsigma_i = \vlb_i + \vub_i$, $\phiij = (\thetaub + \thetalb)/2$, $\deltaij = (\thetaub - \thetalb) / 2$ and {\cblue $\lijbar = \frac{\left |\bs{T}_{ij}\right |^2 \sijbarsq}{\underline{\bs{v}}_i^2}$}. Constraint \eqref{eq: pad} is the phase angle difference constraint, which is a relaxation of the equality $\tan(\bs{\theta}_{ij}) = \frac{w^R_{ij}}{w^I_{ij}}$. Constraints \eqref{eq: ots_lnc1} and \eqref{eq: ots_lnc2} are the ``lifted nonlinear cuts" from \citep{bestuzheva2020convex}, derived using trigonometric identities. Constraints \eqref{eq: l_strenghten} and \eqref{eq: ots_rsoc} use the relationship between current magnitude and power flow to tighten the QC relaxation. \eqref{eq: l_bd} bounds the squared current magnitude. Note that while \eqref{eq: pad}, \eqref{eq: l_strenghten}, and \eqref{eq: l_bd} are the same as their counterparts in the ACOPF model, they are still valid for the ACOTS setting. On the other hand, \eqref{eq: ots_lnc1}, \eqref{eq: ots_lnc2}, and \eqref{eq: ots_rsoc} are modified under the ACOTS case, so that when line $(i,j)$ is switched off, constraints \eqref{eq: ots_lnc1} and \eqref{eq: ots_lnc2} become redundant, {\cblue while constraints \eqref{eq: ots_rsoc} and \eqref{eq: l_bd} ensure that $l_{ij}$ goes to 0}. The use of \eqref{eq: ots_rsoc} is new for the ACOTS QC relaxation.

Putting all the constraints together, we obtain the following QC relaxation for the ACOTS problem: 
\begin{subequations} 
\label{eq: qcots}
\begin{align}
(\textbf{ACOTS-QC}): \min~~ &\eqref{eq: ots_objective}\\
\st~~ &\eqref{eq: ots_s_balance}, \eqref{eq: ots_link} - \eqref{eq: z_var}, \eqref{eq:linearize_sij},\hspace{-1cm}\\
& \eqref{eq:MC-q}, \eqref{eq: ots_cosine}-\eqref{eq:triform}, \eqref{eq: qc_constr}.
\end{align}
\end{subequations}

The relationship between the solution sets of different formulations is simplified and shown in Figure \ref{fig:venn_relax}. Here, ACOPF is the non-convex polar formulation, which is equivalent to the ACOTS model with all the lines switched on; ACOPF-QC is the QC relaxation for ACOPF from \citep{coffrin2015qc}.
\begin{figure}[htbp]
    \centering
    \includegraphics[width=0.3\textwidth]{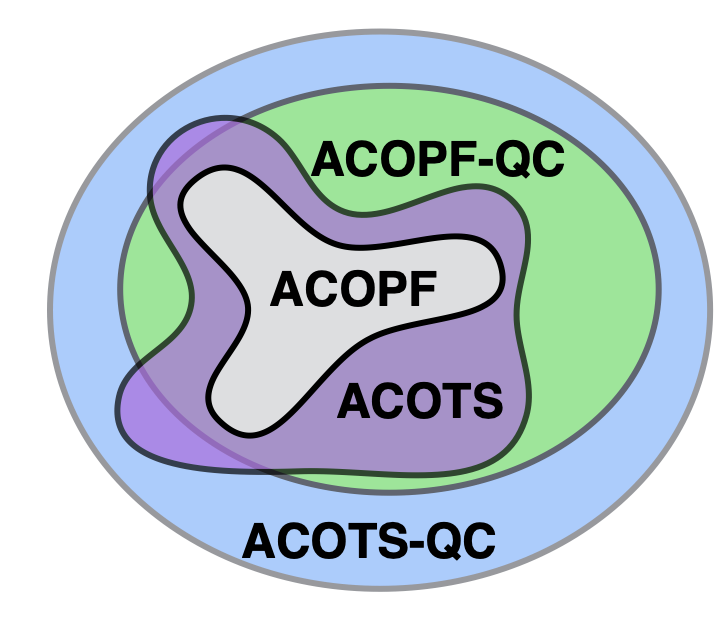}
    \caption{Venn diagram for solution sets of different formulations.}
    \label{fig:venn_relax}
\end{figure}

{\cred Tight convex relaxations for ACOTS such as (ACOTS-QC) can serve several purposes. The solution to a tight relaxation problem is a good approximation for optimal ACOTS decisions. Also, they provide tight lower bounds for evaluating the quality of a feasible ACOTS solution, and can be used to prove global optimality of a feasible solution when the optimality gap is 0. In addition, the convex relaxation can be used in a two-stage process for minimizing the cost of ACOPF, where we first change the network topology using line switching solutions from the relaxation, and then solve the ACOPF problem on the updated network. In what follows, we show how to further tighten (ACOTS-QC) with \cycconstrs \ and OBBT.}
\section{Cycle-Based On/Off Polynomial Constraints}
In this section, we present a novel type of \cycconstrs\ based on lifted trigonometric auxiliary variables $c_{ij}$ and $s_{ij}$. We use both these new \cycconstrs \ and \cycconstrs \ with voltage product variables $w^R_{ij}$, $w^I_{ij}$, and $w_i$ from \citep{kocuk2016strong} to strengthen the on/off QC relaxation for ACOTS. Since those constraints are polynomial functions in multilinear terms, we linearize them with the extreme-point representation. {\cred We also develop novel cutting planes to incorporate those constraints more efficiently via branch-and-cut framework.}
\subsection{Formulating Lifted Cycle Constraints}
The \cycconstrs \ are formulated based on the fact that for any given cycle $\mc{C}$ in the transmission network, the voltage angle differences of all lines in $\mc{C}$ sum up to 0. More formally, let $(\hat{v}_1, \hat{v}_2,\hdots,\hat{v}_n, \hat{v}_1)$ be a vertex sequence for the cycle $\mc{C}$ of length $n$, and let the cycle be represented by its lines: $\mc{C} = \{(\hat{v}_1, \hat{v}_2), (\hat{v}_2, \hat{v}_3),\hdots,(\hat{v}_n, \hat{v}_1)\}$, then we have $\sum_{(i,j)\in\mc{C}} \theta_{ij} = 0$.

The method we use to derive \cycconstrs \ for the QC relaxation is similar to that of \citep{kocuk2016strong} where \cycconstrs \ for the SOC relaxation are derived. However, unlike in the SOC relaxation, the main advantage of the QC relaxation is that we have direct access to both $(c_{ij}$, $s_{ij})$ and $(w^R_{ij},w^I_{ij})$ variables, enabling us to formulate additional \cycconstrs \ that could further enhance the relaxation quality. 

In what follows, we derive \cycconstrs \ for cycles with 3 and 4 nodes. We call the resulting constraints as 3-cycle constraints and 4-cycle constraints, respectively. We first develop those constraints without considering switching (on/off) decisions, and then demonstrate how to reformulate them to include those decisions using a tight big-M formulation.
\subsubsection{3-Cycle Constraints}
For a cycle of three nodes $i$, $j$ and $k$, we have $\theta_{ij} + \theta_{jk} + \theta_{ki} = 0$ or equivalently $\theta_{ik} = \theta_{ij} + \theta_{jk}$, which indicates that $\cos(\theta_{ik}) = \cos(\theta_{ij} + \theta_{jk})$ and $\sin(\theta_{ik}) = \sin(\theta_{ij} + \theta_{jk})$. Expanding the right-hand sides and replacing the trigonometric functions with their corresponding lifted variables, we get the following nonlinear 3-cycle constraints:
\begin{subequations}
\label{eq:3_cyc_12}
\begin{align}
c_{ik} &= c_{ij}c_{jk} - s_{ij}s_{jk}\label{eq:3_cyc_1_1} \\
s_{ik} &= c_{ij}s_{jk} + s_{ij}c_{jk}, \label{eq:3_cyc_1_2}
\end{align}
\end{subequations}
Though simple, these are novel and are applicable to both ACOPF and ACOTS problems. We can then obtain the \cycconstrs \ in \citep{kocuk2016strong} by multiplying both sides of \eqref{eq:3_cyc_1_1} and \eqref{eq:3_cyc_1_2} with $v_i v_j^2 v_k$, and using the relationships in \eqref{eq: ots_wii} and \eqref{eq:w_nonconv} (ignoring switching decisions in \eqref{eq:w_nonconv} for now):
\begin{subequations}
\label{eq:3_cyc_34}
\begin{align} 
w_jw^R_{ik} &= w^R_{ij}w^R_{jk} - w^I_{ij}w^I_{jk}, \\  
w_jw^I_{ik} &= w^R_{ij}w^I_{jk} + w^I_{ij}w^R_{jk}.
\end{align}
\end{subequations}
Alternatively, constraints \eqref{eq:3_cyc_34} can also be derived from minor-based reformulation as in \citep{kocuk2018matrix}.

For brevity, in the following, we call \cycconstrs \ with $c_{ij}$ and $s_{ij}$ variables as \cycconstrs \ \textit{in the $c$-$s$ space}, and \cycconstrs \ with $w^R_{ij}$, $w^I_{ij}$, and $w_i$ variables as \cycconstrs \ \textit{in the $w$ space}. 

We also derive two more sets of \cycconstrs \ by considering other permutations of the equality $\theta_{ij} + \theta_{jk} + \theta_{ki} = 0$, including $\theta_{jk} + \theta_{ki} = \theta_{ji}$ and $\theta_{ki} + \theta_{ij} = \theta_{kj}$. 
Though they look equivalent in the nonlinear forms, the linearized versions (see Section \ref{sec: cycle_ex_form}) of these permutated constraints do not necessarily dominate one another, and we add all of them to tighten the relaxation. 

Another type of \cycconstrs \ can be derived from the equality $\theta_{ij} + \theta_{jk} - \theta_{ik} = 0$, which leads to the following constraints in the $c$-$s$ space:
\begin{subequations}\label{eq:3_cyc_v2}
\begin{alignat}{4}
c_{ij} c_{jk} c_{ik} + c_{ij} s_{jk} s_{ik} - s_{ij} s_{jk} c_{ik} + s_{ij} c_{jk} s_{ik} &= 1 \label{eq:3_cyc_v2_1}\\
s_{ij} c_{jk} c_{ik} + s_{ij} s_{jk} s_{ik} + c_{ij} s_{jk} c_{ik} - c_{ij} c_{jk} s_{ik} &= 0. \label{eq:3_cyc_v2_2}
\end{alignat}
\end{subequations}

The following proposition shows the equivalence between constraints \eqref{eq:3_cyc_12} and \eqref{eq:3_cyc_v2}. This proposition is similar to the Proposition 4.1 in \citep{kocuk2016strong}, but our result is in the $c$-$s$ space rather than the $w$ space, and we simplify the presentation of the result. Also, we provide a new way to prove this result.
\begin{proposition}
For all $(i,j)\in\mc{A}$, if $c_{ij}$ and $s_{ij}$ satisfy $c_{ij}^2 + s_{ij}^2 = 1$, then $\{(c, s): \eqref{eq:3_cyc_12} \text{ holds}\} = \{(c,s): \eqref{eq:3_cyc_v2} \text{ holds}\}$.
\end{proposition}

\proof{Proof.}
We first prove that $\{(c, s): \eqref{eq:3_cyc_12} \text{ holds} \}\subseteq \{(c,s): \eqref{eq:3_cyc_v2} \text{ holds}\}$. If the variables $c$ and $s$ satisfy \eqref{eq:3_cyc_1_1}-\eqref{eq:3_cyc_1_2}, we multiply both sides of \eqref{eq:3_cyc_1_1} with $c_{ik}$ and both sides of \eqref{eq:3_cyc_1_2} with $s_{ik}$, then sum up those two equations: 
\[
c_{ik}^2 + s_{ik}^2 = c_{ik} c_{ij} c_{jk} - c_{ik} s_{ij} s_{jk} + s_{ik} c_{ij} s_{jk} + s_{ik} s_{ij} c_{jk}.
\]

Since the left-hand side is equal to 1, this equation is equivalent to constraint \eqref{eq:3_cyc_v2_1}. Similarly, we obtain constraint \eqref{eq:3_cyc_v2_2} by multiplying both sides of \eqref{eq:3_cyc_1_1} with $s_{ik}$ and both sides of \eqref{eq:3_cyc_1_2} with $c_{ik}$, and deduct the second equation from the first.

For the reverse direction, if $c$ and $s$ satisfy \eqref{eq:3_cyc_v2}, let $a = c_{ij}c_{jk} - s_{ij}s_{jk}$ and $b = c_{ij}s_{jk} + s_{ij}c_{jk}$, we can rewrite \eqref{eq:3_cyc_v2_1} and \eqref{eq:3_cyc_v2_2} as $c_{ik}a + s_{ik} b = 1$ and $c_{ik}b - s_{ik} a = 0$, respectively.
Solving for $a$ and $b$, we get $a = c_{ik}$ and $b = s_{ik}$, which are equivalent to constraints \eqref{eq:3_cyc_1_1}-\eqref{eq:3_cyc_1_2}.\Halmos
\endproof
\subsubsection{4-Cycle Constraints}
For a 4-cycle with nodes $\{i, j, k, l\}$, we similarly derive \cycconstrs \ based on the equality 
$\theta_{ij} + \theta_{jk} + \theta_{kl} + \theta_{li}= 0$. More specifically, for the permutation $\theta_{ij} + \theta_{kl}= \theta_{il} - \theta_{jk}$, we derive the following \cycconstrs \:
\begin{subequations}\label{eq: 4-cylce}
\begin{align}
& c_{ij}c_{kl} - s_{ij}s_{kl} = c_{il}c_{jk} + s_{il}s_{jk}\\
& c_{ij}s_{kl} + s_{ij}c_{kl} = - c_{il}s_{jk} + s_{il}c_{jk} \\
& w^R_{ij}w^R_{kl} - w^I_{ij}w^I_{kl} = w^R_{il}w^R_{jk} + w^I_{il}w^I_{jk}\\
& w^R_{ij}w^I_{kl} + w^I_{ij}w^R_{kl} = - w^R_{il}w^I_{jk} + w^I_{il}w^R_{jk}.
\end{align}
\end{subequations}

For the other two permutations, i.e., $\theta_{ij} + \theta_{jk} = \theta_{il} - \theta_{kl}$ and $\theta_{jk} + \theta_{kl} = \theta_{il} - \theta_{ij}$, we can derive similar \cycconstrs. Note that for those permutations, the \cycconstrs \ in the $w$ space contain trilinear terms, thus we do not include those constraints in our implementation for efficiency purposes.
\subsubsection{On/off cycle constraints for ACOTS}
To reformulate \cycconstrs \ for ACOTS and include switching decisions, we use the big-M formulation. As an example, we demonstrate the formulation on constraints \eqref{eq:3_cyc_12}, i.e., the 3-cycle constraints in the $c$-$s$ space. Constraints \eqref{eq:3_cyc_12} are only valid when all lines in the 3-cycle $\mc{C}$ are switched on, which is ensured by the following big-M constraints:
\begin{subequations}
\label{eq:ots_3_cyc_12}
\begin{align}
&- 3 \widehat{z} \leqslant c_{ik} - c_{ij}c_{jk} + s_{ij}s_{jk} \leqslant 3 \widehat{z}\\
&- 3 \widehat{z} \leqslant s_{ik} - c_{ij}s_{jk} - s_{ij}c_{jk}  \leqslant 3 \widehat{z}.
\end{align}
\end{subequations}
Here, $\widehat{z} = \sum_{(l,m)\in \mc{C}} (1 - z_{lm})$ and we use ``3" as the big-M constant. It is valid because $c_{ij} \in [0, 1]$ and $s_{ij} \in [-1, 1]$ for the worst-case bounds of $\theta_{ij} \in [-\frac{\pi}{2},\frac{\pi}{2}]$. Although, this could be further improved if the angle-difference bounds are tighter. We also include similar on/off constraints for all 4-cycles with appropriate big-M constants.  
\subsection{Extreme-Point Representation}\label{sec: cycle_ex_form}
The \cycconstrs \ contain bilinear terms, which are usually linearized with McCormick relaxation in the literature \citep{kocuk2016cycle, kocuk2018matrix}. We instead use the extreme-point representation to linearize those constraints, which is \textit{guaranteed to capture the convex hull} of the \cycconstrs \ for a given cycle (including all permutations in the $c$-$s$ or $w$ space). 

For example, let
$x_i^{\mathbf{c}} \ \forall i=1,\hdots, 6$ represent variables $c_{ij}, c_{jk}, c_{ik}, s_{ij}, s_{jk}, s_{ik}$, respectively. We first rewrite 3-cycle constraints \eqref{eq:3_cyc_12} and {\cblack its counterparts by permutation} as follows: 
\begin{subequations}\label{eq: 3_cyc_lifted}
\begin{align}
    & x_3^{\mathbf{c}} = x_1^{\mathbf{c}} x_2^{\mathbf{c}} - x_4^{\mathbf{c}} x_5^{\mathbf{c}}, \quad 
    x_6^{\mathbf{c}} = x_1^{\mathbf{c}} x_5^{\mathbf{c}} + x_2^{\mathbf{c}} x_4^{\mathbf{c}} \\
    & x_1^{\mathbf{c}} = x_2^{\mathbf{c}} x_3^{\mathbf{c}} + x_5^{\mathbf{c}} x_6^{\mathbf{c}}, \quad
    x_4^{\mathbf{c}} = x_2^{\mathbf{c}} x_6^{\mathbf{c}} - x_3^{\mathbf{c}} x_5^{\mathbf{c}} \\
    & x_2^{\mathbf{c}} = x_1^{\mathbf{c}} x_3^{\mathbf{c}} + x_4^{\mathbf{c}} x_6^{\mathbf{c}}, \quad
    x_5^{\mathbf{c}} = x_1^{\mathbf{c}} x_6^{\mathbf{c}} - x_3^{\mathbf{c}} x_4^{\mathbf{c}}.
\end{align}
\end{subequations}

Let binary variable $\zc$ equal 1 if and only if all lines in the cycle $\mc{C} = \{(i,j), (j, k), (k, i)\}$ are switched on. $\xc_{j_1 j_2}$ is a lifted variable for $\xc_{j_1} \xc_{j_2}$. We can linearize the constraint $x_3^{\mathbf{c}} = x_1^{\mathbf{c}} x_2^{\mathbf{c}} - x_4^{\mathbf{c}} x_5^{\mathbf{c}}$ and connect the constraint with line switching decisions as follows:
\begin{align}\label{eq: ots_cyc3_conv2}
\min(0, \xclb_3)(1 - \zc) \leqslant x_3^{\mathbf{c}} - x_{12}^{\mathbf{c}} + x_{45}^{\mathbf{c}} \leqslant \max(0, \xcub_3)(1 - \zc).
\end{align}

We will explain this constraint in more detail at the end of this section after introducing constraints \eqref{eq: cyc3_conv}. Other constraints in \eqref{eq: 3_cyc_lifted} can be linearized in a similar way.

In addition to the linearization above, we have the following constraints in the extreme-point representation for constraints \eqref{eq: 3_cyc_lifted} (with switching decisions added): 
\begin{subequations}\label{eq: cyc3_conv}
\begin{align}
    & \sum_{i=1}^{64} \lambda^{cs}_i = \zc \label{eq:ots_cyc_3_conv_4_5_1} \\
    & \lambda^{cs}_i \geqslant 0  && \forall i=1,\hdots, 64 \label{eq:ots_cyc_3_conv_4_5_2} \\
    & x_j^{\mathbf{c}} \geq \xclb_j \left(\sum_{i: (\mathcal{X}^i_j = \xclb_j)} \lambda^{cs}_i \right)+ \xcub_j \left(\sum_{i: (\mathcal{X}^i_j = \xcub_j)} \lambda^{cs}_i\right) + \min(0, \xclb_j)(1 - \zc) && \forall j=1,\ldots,6 \label{eq: ots_cyc_3_conv_6_1}\\
    & x_j^{\mathbf{c}} \leqslant \xclb_j \left(\sum_{i: (\mathcal{X}^i_j = \xclb_j)} \lambda^{cs}_i \right)+ \xcub_j \left(\sum_{i: (\mathcal{X}^i_j = \xcub_j)} \lambda^{cs}_i\right) + \max(0, \xcub_j)(1 - \zc)  && \forall j=1,\ldots,6 \label{eq: ots_cyc_3_conv_6_2}\\
    & \xc_{j_1 j_2} = \sum_{i=1}^{64} \lambda^{cs}_i \left( \mathcal{X}^i_{j_1} \mathcal{X}^i_{j_2} \right)&& \forall (j_1, j_2)\in \mc{P}\label{eq: cyc_3_conv_7}\\
    & 1 - \sum_{(i,j)\in\mc{C}} (1 - z_{ij}) \leqslant \zc \leqslant \frac{1}{|\mc{C}|} \sum_{(i,j)\in\mc{C}} z_{ij}\label{eq: zc_z_1}\\
    & \zc \in\{0,1\}\label{eq: zc_z_2}
\end{align}
\end{subequations}
where $\lambda^{cs}_i$ is an auxiliary variable. $\mc{P} =$ \{(1,2), (1,3), (1,5), (1,6), (2,3), (2,4), (2,6), (3,4), (3,5), (4,5), (4,6), (5,6)\}. $\mathcal{X}$ can be viewed as a matrix of size $2^6 \times 6$, such that every row represents all possible combinations of the lower and upper bounds of variables $x_j^{\mathbf{c}} \in [\xclb_j,\xcub_j] \ \forall j=1,\hdots,6$. Constraints \eqref{eq:ots_cyc_3_conv_4_5_1} and \eqref{eq:ots_cyc_3_conv_4_5_2} set bounds for auxiliary multiplier variables. When $\zc = 1$, constraints \eqref{eq: ots_cyc_3_conv_6_1}, \eqref{eq: ots_cyc_3_conv_6_2}, and \eqref{eq: cyc_3_conv_7} represent the convex hull consisting of variables $\xc_j~(\forall j)$ and $\xc_{j_1 j_2}~(\forall (j_1, j_2))$; when $\zc = 0$ constraints \eqref{eq: ots_cyc_3_conv_6_1} and \eqref{eq: ots_cyc_3_conv_6_2} become redundant. Constraint \eqref{eq: zc_z_1} connects $\zc$ and $z_{ij}$: when all lines are switched on, \eqref{eq: zc_z_1} fixes $\zc$ to 1. If any line is switched off, $1 - \sum_{(i,j)\in\mc{C}} (1 - z_{ij}) \leqslant 0$ and $\frac{1}{|\mc{C}|} \sum_{(i,j) \in\mc{C}} z_{ij} \in [0, 1)$, which enforce $\zc = 0$. Also note that when $\zc = 0$, \eqref{eq:ots_cyc_3_conv_4_5_1}, \eqref{eq:ots_cyc_3_conv_4_5_2} and \eqref{eq: cyc_3_conv_7} ensure $x_{12}^{\mathbf{c}} = x_{45}^{\mathbf{c}} = 0$, and constraint \eqref{eq: ots_cyc3_conv2} becomes redundant. We can similarly derive the convex hull formulation for 3-cycle constraints in the $w$ space and for 4-cycle constraints. 

{\cred
To show the tightness of extreme-point formulation compared with McCormick relaxation, we use the scatter plot method, similar to that of \citep{luedtke2012some}. We apply the two different relaxation methods for the summation of bilinear terms $\sum_{(j_1, j_2)\in\mc{P}} \xc_{j_1} \xc_{j_2}$. Without loss of generality, we set the domain of $\xc_{i}, i=1,...,6$ as $[-1, 1]$, which include both positive and negative numbers. We randomly generate 5,000 samples of those points, following a uniform distribution in their domain. For each sample, we then obtain the difference between upper and lower bounds of the summation with the two relaxation methods, and obtain the scatter plot in Figure \ref{fig: scatter_cyc}. In the scatter plot, each point (i.e., blue spot) corresponds to the result of one sample. All the points are above the (grey) diagonal line, which implies that the extreme-point representation is either tighter or as tight as the McCormick relaxation.}

\begin{figure}[htbp]
    \centering
    \includegraphics[width=0.4\textwidth]{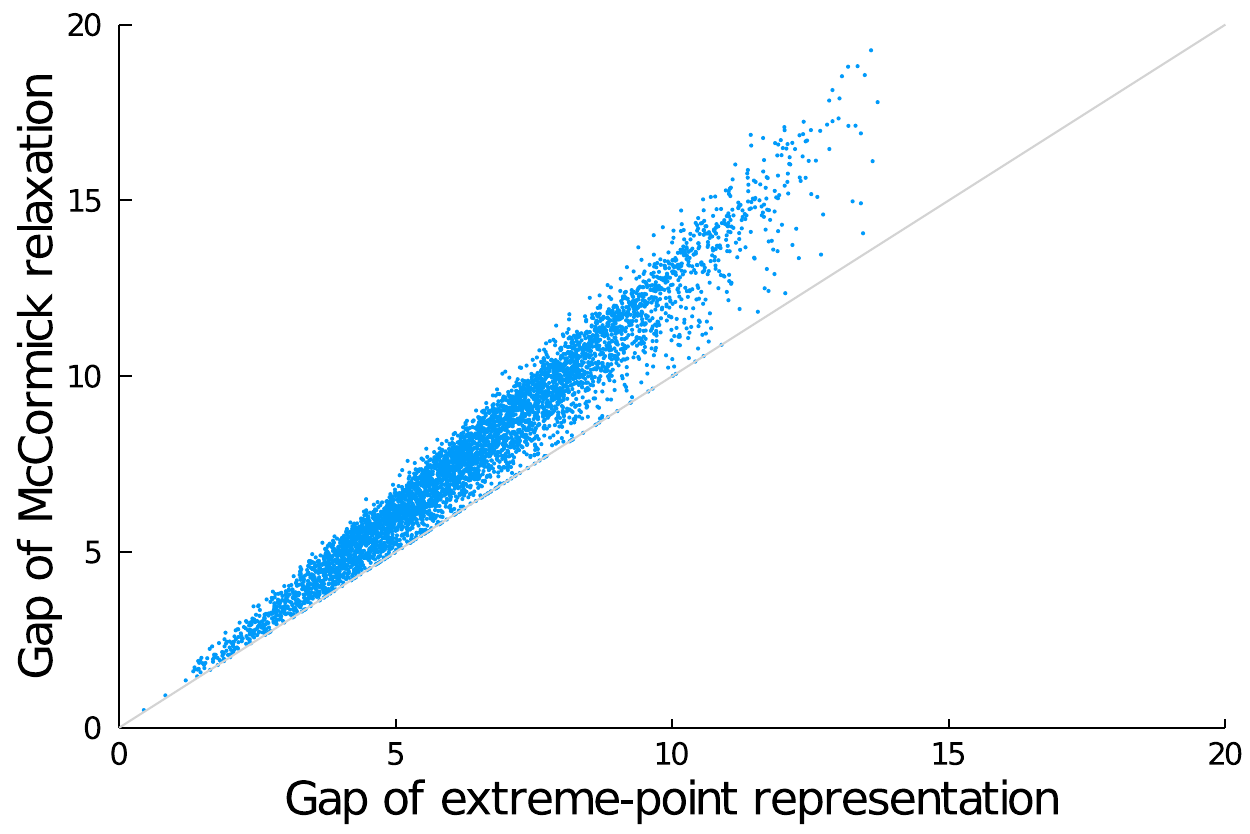}
    \caption{\small Scatter plot comparing extreme-point formulation with McCormick relaxation for summation of bilinear terms.}
    \label{fig: scatter_cyc}
\end{figure}

\subsection{Branch-and-Cut Algorithm for Lifted Cycle Constraints}
\label{subsec:b_and_c}
{\cred 
The size of extreme-point formulation for \cycconstrs \ grows quickly with the number of cycles in the network. Therefore, instead of adding all of those constraints at once, we use a separation scheme which generates cutting planes only when the linearized \cycconstrs \ are violated. 

Due to binary line switching decisions in the linearized \cycconstrs, the separation problem is not a linear program (LP). To generate cutting planes that separates infeasible solutions, we first ignore the line switching decisions in the linearized \cycconstrs \ and generate Benders feasibility cuts if those constraints are violated. We then incorporate binary variables into the Benders cuts via disjunctive programming.

First, we describe how to generate Benders cuts without considering line switching decisions. At the start of the cutting-plane algorithm we solve the ACOTS-QC model without any \cycconstrs, and obtain optimal solutions for $c$-$s$ and $w$ variables. Then for each cycle we solve a feasibility problem consisting of all the \cycconstrs \ (within the $c$-$s$ or $w$ space) in the extreme-point formulation without line switching decisions, while fixing $c$-$s$ or $w$ variables to their optimal values. If this feasibility problem is feasible, then none of the linearized \cycconstrs \ is violated, so we do not need to add any cut; otherwise, we generate a Benders feasibility cut, and add this cut back to the ACOTS-QC model and solve it again. The algorithm terminates if at one iteration the ACOTS-QC model solution is feasible to the \cycconstrs \ for all cycles.

For example, for a 3-cycle with nodes $i$, $j$, and $k$, let $\xc = (c_{ij}, c_{jk}, c_{ik}, s_{ij}, s_{jk}, s_{ik})$ and let $\xcstar$ be an optimal solution of the ACOTS-QC model. We solve the following separation problem: 
\begin{subequations}\label{eq: separation}
\begin{alignat}{3}
    \min~~&0\\
    \st~~&\text{extreme-point representation of \eqref{eq: 3_cyc_lifted}} \\
    & \xc = \bs{x}^{\mathbf{c}*}.
\end{alignat}
\end{subequations} 
If this problem is infeasible, we can generate a Benders feasibility cut in the following form:
\begin{align}\label{eq: acopf_benders}
    \betabar^\top \xc \leq \bs{b}
\end{align}
where $\betabar$ and $\bs{b}$ are the coefficient vector and the constant in the Benders cut, respectively.
Now we consider the impact of line switching decisions. The Benders cut should be redundant when any line in a cycle is turned off. To ensure this, we reformulate Benders cuts via disjunctive programming. Again, we use the 3-cycle constraints in the $c$-$s$ space to demonstrate how this works. Remember the binary variable $\zc$ that equals 1 if and only if all lines in the cycle $\mc{C}$ are turned on. The Benders cut \eqref{eq: acopf_benders} should only be active when $\zc=1$. In other words, the feasible region defined by the reformulated Benders cut is a union of the following two sets:
\begin{align*}
    &\Gamma_0 = \{(\xc, \zc): \zc = 0,~ \min(0, \xclb) \leq \xc \leq \max(0, \xcub) \}\\
    & \Gamma_1 = \{(\xc, \zc): \zc = 1,~\betabar^\top \xc \leq \bs{b}\}
\end{align*}
where $\xclb$ and $\xcub$ are lower and upper bounds of $\xc$. Using disjunctive programming, the following constraint is valid for the convex hull of $\Gamma_0\cup \Gamma_1$:
\begin{align}\label{eq: acots_benders}
    \betabar^\top \xc \leq \zc \bs{b} + (1 - \zc)\Bigg(\sum_{i\in\mc{N}: \betabar_i < 0} \betabar_i\min(0, \xclb_i) + \sum_{i\in\mc{N}: \betabar_i >0} \betabar_i\max(0, \xcub_i)\Bigg)
\end{align}
where $\mc{N}$ is the set of indices for $\betabar$ and $\xc$. Intuitively, when $\zc = 1$, \eqref{eq: acots_benders} is exactly the Benders cut \eqref{eq: acopf_benders}; when $\zc = 0$, \eqref{eq: acots_benders} is always satisfied and thus becomes redundant. 

The separation scheme needs to solve the mixed-integer quadratic ACOTS-QC model at each iteration, which is very inefficient. This is why we use a branch-and-cut method, where the mixed-integer quadratic ACOTS-QC model is only solved once with the branch-and-bound algorithm, and the Benders cuts are added at integral nodes of the branch-and-bound tree. More specifically, at each integral node, we obtain the optimal values of trigonometric terms $\xc$ and switching decisions $z_{ij}$, and denote them respectively by $\xcstar$ and $\zstar$. For any cycle $\mc{C}$ with all lines turned on (i.e., when $\sum_{(i,j)\in\mc{C}} \zstar = |\mc{C}|$), we solve the separation problem \eqref{eq: separation}. If the problem is infeasible, we add constraints \eqref{eq: zc_z_1}, \eqref{eq: zc_z_2}, and \eqref{eq: acots_benders}.}

\section{Optimization-Based Bound Tightening}
\label{subsec:obbt}
The OBBT method is a technique in non-convex optimization, which aims to improve the convex relaxation bound by tightening the bounds of certain variables. OBBT is often used to improve bounds and obtain near global-optimal solutions for ACOPF problems \citep{chen2015bound,cengil2022learning}, and it has the benefit of being massively parallelizable \citep{gopinath2020proving}. In our work, we implement OBBT to tighten the bounds of $v_i$, $\theta_{ij}$, $z_{ij}$, and $y_{C}$ variables before solving the relaxations of ACOTS.

To formulate bound tightening optimization models for any variable $x$, we replace the objective of an ACOTS relaxation (e.g., ACOTS-QC) with $\max x$ or $\min x$.
To avoid solving time-consuming MINLPs in OBBT, we linearly relax all integer variables. 
We denote the optimal objectives of the bound tightening maximization and minimization problems $\bar{x}$ and $\underline{x}$. If $x$ is a binary variable (such as $z_{ij}$ and $y_{C}$), we can further tighten their bounds by fixing $x$ to 1 if $\underline{x} > 0$, and to 0 if $\bar{x} < 1$ within the OBBT iteration. 
The OBBT algorithm terminates when the bounds of all variables stop improving, or when the algorithm reaches its time/iteration limit.

{\cblue 
\section{Maximum Spanning Tree Heuristic}\label{sec: span_tree}
Heuristics have been developed to quickly find good feasible solutions for the OTS problem, prioritizing speed over optimality \citep{fuller2012fast, barrows2012computationally, soroush2013accuracies, hinneck2022optimal}. These approaches typically reduce the solution space by fixing a set of candidate lines based on objective sensitivities. However, they often apply heuristics to the simpler, less accurate DCOTS model for large-scale networks or focus on quantifying the accuracy of such decisions. None have explored the problem using more accurate ACOTS relaxations, particularly for large-scale networks.

In this section, we propose a heuristic to accelerate the solution of the {\bf (ACOTS-QC)} formulation in \eqref{eq: qcots}, where we find a maximum spanning tree in the network, and keep all lines in the spanning tree switched
on. Note that a maximum spanning tree is a spanning tree of a weighted graph with the maximum total weight.

Based on empirical observations, most transmission lines remain switched on in the optimal solution. Inspired by this observation, we heuristically identify the lines that are likely the most essential for transmission and that maintain network connectivity, keeping these lines switched on to reduce the number of binary variables. Specifically, we assign the weight $\max\left(\frac{\left|\widehat{S}_{ij}\right |^2}{\cblue \sijbarsq}, \frac{\left |\widehat{S}_{ji} \right |^2}{\cblue \sijbarsq}\right)$ to each line $(i,j)$, where $\widehat{S}_{ij}$ is the locally optimal solution corresponding to the AC power flow variable, $S_{ij}$, in the ACOPF problem (e.g., obtained using Ipopt) based on the original topology with all lines active. Intuitively, lines with higher weights exhibit lower slack. We then find a maximum spanning tree for the network and allow only lines outside the tree to be switched off.

This heuristic restricts the solution space of switching variables, potentially leading to sub-optimal solutions for \eqref{eq: qcots}. However, even sub-optimal line-switching can reduce grid operating costs. Our numerical experiments in Section \ref{sec: experiment_heuristic} demonstrate that the heuristic can quickly find near-optimal solutions for many instances, particularly larger ones.
Additionally, in Section \ref{sec: large_inst}, we show that the heuristic obtains results for several large instances with 500 to 2,312 buses, which, to our knowledge, have not been previously explored for ACOTS with tight convex relaxations. 
}

\section{\cblue Overview of Proposed Algorithm}
\label{ch: overview}
{\cred {\cblue In this section, we provide a summary of our algorithm with all proposed improvements (with lifted cycle constraints added in branch-and-cut fashion).

We start the algorithm by preprocessing variable bounds via OBBT as described in Section \ref{subsec:obbt}. Those bounds are added to (\textbf{ACOTS-QC}), which is then solved via branch-and-bound. Inside the branch-and-bound search, lifted cycle constraints are generated as described in Section \ref{subsec:b_and_c} and added at integral nodes via lazy callback of the solver. The algorithm terminates when the gap between the upper and lower
bounds of the branch-and-bound search is below a small tolerance. Finally, if we opt to use the maximum spanning tree heuristic, it can be incorporated before OBBT.} We summarize the implementation of the algorithm in Figure \ref{fig: flowchart}. 

\bigskip
\begin{figure}[htbp] 
\begin{center}
\vspace*{-0.4cm}
{
\begin{tikzpicture}[node distance=1.75cm]
\tikzstyle{startstop} = [rectangle, rounded corners, minimum width=1.15cm, minimum height=0.7cm,text centered, draw=black, fill = gray!40!white]
\tikzstyle{process} = [rectangle, minimum width=2cm, minimum height=1cm, text centered, text width=1.98cm, draw=black]
\tikzstyle{decision} = [diamond, aspect=1.4, text centered, text width=1.2cm, draw=black]
\tikzstyle{arrow} = [thick,->,>=stealth]

\small
\node (start) [startstop] {START};
\node (obbt) [process, right of=start,xshift = 0.7cm,text width=2.5cm] {Tighten variable bounds via OBBT};
\node (span) [process, below of=obbt,yshift = -0.4cm,xshift = -1.55cm, text width=2.5cm] {\cblue (optional) Maximum Spanning Tree Heuristic};
\node (lp) [process, right of=obbt,xshift = 1.6cm,text width=3cm] {Solve node linear relaxation in branch-and-bound};

\node (trans) [right of=lp, xshift = 2cm] {};
\node (dummytrans) [xshift = 2cm] at (trans.south) {};

\node (stopping) [decision, below of=dummytrans,aspect=1.2,yshift = -1cm, text width=1.47cm] {Stopping criterion?};
\node (stop) [startstop, right of=stopping,xshift = 1.5cm] {STOP};
\node (int) [decision, aspect=1.33,text width=1.4cm] at (lp |- stopping) {Check \\ integrality};

\node (lazy) [process, below of=int,yshift = -1.1cm] {Apply lazy callback};
\node (cuts) [process, right of= lazy, xshift = 2cm, text width=2cm] {Generate cycle-based cuts};
\node (dummy3) [right of=lazy, xshift=0.35cm] {};

\coordinate (mid) at ($(start)!0.35!(obbt)$);

\draw [arrow] (start) -- (obbt);
\draw [arrow] (span) -- (mid);
\draw [arrow] (obbt) -- (lp);
\draw [arrow] (stopping) -- node[anchor=east,xshift = 0.2cm,yshift=0.2cm] {Yes} (stop);
\draw [arrow] (int) -- node[anchor=south,xshift = 0.1cm,yshift=0.01cm, text width = 2.7cm] {$z_{ij},\forall (i,j)\in\mc{A}$ are not all integral} (stopping);

\draw [arrow] (int) -- node[anchor=south,xshift = 1.5cm,yshift=-0.36cm, text width = 2.6cm] {$z_{ij},\forall (i,j)\in\mc{A}$ are integral} (lazy);
\draw [arrow] (stopping) |- node[anchor=north,xshift = 0.4 cm,yshift=-1cm] {No} (lp);
\draw [arrow] (cuts) -| (stopping);
\draw [arrow] (lazy) -- (cuts);
\draw [arrow] (lp) -- (int);
\end{tikzpicture}
}
\caption{\cred Flow chart of the proposed algorithm. }
\label{fig: flowchart}
\end{center}
\end{figure}
}

{\cblue Combining all strengthening techniques in Figure \ref{fig: flowchart} provides the tightest QC-based ACOTS relaxation in the literature. Note that when strengthening the ACOTS relaxations, two factors are crucial: (i) tightness of the MINLP convexification, and (ii) tightness of the corresponding ACOPF relaxation. The later is important because if the binary variables are known, then the ACOTS problem simplifies to an ACOPF problem. In our work, several approaches, such as disjuctive programming-based lifted cycle constraints, are employed to strengthen the MINLP bound. While some of our proposed methods, such as the lifted cycle constraints, can also improve the ACOPF relaxation, as shown in Section \ref{ch: cycle_acopf}, our main focus is not on improving the ACOPF bound, which remains challenging on a large-scale and warrants further exploration. 

}

\section{Numerical Experiments}\label{sec: expriments}
This section presents the numerical efficacy of the proposed ACOTS-QC and ACOPF-QC relaxations with \cycconstrs, and an analysis for ACOTS with different load profiles. Our experiments are conducted on PGLib-OPF v20.07 benchmark library \citep{babaeinejadsarookolaee2019power}. {\cblue For medium-scale instances, we }use a Linux workstation with 3.6GHz Intel Core i9-9900K CPUs and 128GB memory{\cblue ; For large-scale instances in Section \ref{sec: large_inst}, we use Linux workstations with Intel CPUs and 250GB memory}. The programming language is Julia v1.6. We locally solve all non-convex MINLP (ACOTS) {\cblue formulations with both Juniper.jl (v0.7.0) \citep{kroger2018juniper} and Knitro (v.13.0.1) \citep{byrd2006k}. Non-convex NLP (ACOPF) formulations are solved locally using Ipopt (v3.13.4) \citep{wachter2006implementation}.} All relaxation formulations (ACOTS-QC, ACOPF-QC and OBBT iterations) are solved using the Gurobi (v9.0.0) solver. The branch-and-cut framework for cycle constraints is implemented using Gurobi's lazy-constraint callback. { The code for our experiments are available in the INFORMS Journal on Computing GitHub software repository \citep{guo2025acots}. }

\subsection{Relaxations for ACOTS}\label{sec: experiment_acots}
We compare five different types of relaxations for ACOTS:

(1) {\cblue``P"}: The on/off QC relaxation implemented in
PowerModels.jl \citep{coffrin2018power}, which is used as state-of-the-art to benchmark ACOTS relaxations. Formulation within {\cblue``P"} is based on \citep{hijazi2017convex} which uses on/off trigonometric function relaxations and recursive McCormick linearization of trilinear terms, without additional cycle constraints or the OBBT algorithm.  

(2) ``E": Proposed ACOTS-QC relaxation with extreme-point representation for linearizing 
$z_{ij}v_i v_j c_{ij}$ and $z_{ij}v_i v_j s_{ij}$ in \eqref{eq:w_nonconv}. 

(3) ``EC": Tightened ``E" with \cycconstrs. {\cblue The \cycconstrs~are included in (\textbf{ACOTS-QC}) at the start of the algorithm.}

(4) ``ECB": Includes all proposed improvements (extreme-point representation, \cycconstrs, and OBBT). {\cblue The \cycconstrs~are included in (\textbf{ACOTS-QC}) at the start of the algorithm (Compare with ``ECB*" below).}

(5) ``ECB*": The same as ``ECB", except the \cycconstrs~are added via branch-and-cut framework as in Section \ref{subsec:b_and_c}. {\cblue The number of added cuts is bounded at 200.}   

We also provide initial feasible solutions as {\it warm-start} solutions, which are helpful to speedup the convergence for many of the large instances. Those initial feasible solutions are obtained by solving ACOPF-QC relaxation with recursive McCormick linearization for trilinear terms (for {\cblue ``P"}), ACOPF-QC (with extreme-point linearization, for ``E") or ACOPF-QC with \cycconstrs \ (for ``EC"), and those solutions are valid when all lines in the network are switched on. 

We run PGLib instances with up to 300 buses under typical operating conditions (TYP), as well as cases with small angle-difference conditions (SAD) and congested operating conditions (API) {\dblue , which are all created by \citep{babaeinejadsarookolaee2019power}. More specifically, the TYP cases are base cases with no change in the standard PGLib-OPF networks. The SAD cases modify the TYP cases by reducing the voltage angle difference on all branches of the standard networks, while the API cases modify the TYP cases by increasing active power demand throughout the standard networks.}

In Table \ref{tab: ots} we present results for cases that are solved in the 2-hour time limit (within 0.1\% optimality tolerance) for ``E". {\cblue As shown in Table \ref{tab: rel_gap_ots} of the Online Supplement, which provides the relative gap at termination for all PGLib instances up to 300 buses, instances not solved by ``E" are also unsolved by other methods.} The performance measures we use for comparison include optimality gap and runtime. We put ``ns." for the optimality gaps of cases that are not solved to 0.1\% optimality tolerance within the time limit, and ``tl." for the runtime of test cases that hit the time limit. 

The optimality gap is calculated by (UB - LB)/UB$\times$100 where LB is the optimal value from relaxations of ACOTS, and UB is an upper bound for ACOTS. For UB, we take the minimum of \textit{local optimal values} of the following {\cblue four types of methods}:
\BI
\I {\cblue Non-convex ACOTS model \eqref{eq:ACOTS} with the Knitro solver.}
\I Non-convex ACOTS model \eqref{eq:ACOTS} {\cblue with the Juniper solver}. 
\I Non-convex ACOPF model with all lines switched on. 
\I Non-convex ACOPF model with the set of lines switched off, as indicated by the ACOTS-QC solutions. 
\EI

We highlight with boldface the optimality gaps improved after the relaxations are tightened. All comparisons are between two adjacent columns in the table. 
{\cblack We also highlight the reduced runtimes of our branch-and-cut algorithm (in ECB*).}

The runtimes of ``ECB", and ``ECB*" are the runtimes of the ACOTS-QC relaxation problems and do not include the runtimes of OBBT. This is because OBBT time is a constant factor inclusion irrespective of whether the cycle constraints are added to LP-relaxed models ($z_{ij}, y_{C} \in [0,1])$, directly or in a branch-and-cut fashion within the OBBT algorithm. Moreover, these times are not as significant when compared with the ACOTS-QC relaxation problems, as the OBBT's LP-relaxed models, at every iteration, can be solved in parallel. 
{\cblack We also exclude the model building time of {\cblue separation problems} in ``ECB*" within the branch-and-cut algorithm {\cblue (we separately provide the separation problem building time in Table \ref{tab: sep_build} of the Online Supplement)}, as any overhead in such time is an artifact of the mathematical modeling package within Julia. {\cblue We set the feasibility tolerance for { ``EC" and }``ECB*" to $10^{-4}$, instead of Gurobi's default of $10^{-6}$, because with a smaller tolerance the solver may incorrectly reject feasible warm-start solutions for some instances.}
In addition, we set the upper bound on the number of added cuts at 200, as adding too many cuts could slow down the performance. We observe that those added cuts are able to significantly improve the bounds as shown in Table \ref{tab: ots}.}
\begin{table*}[htbp]
\centering
\caption{Optimality gap and runtime of ACOTS relaxations (bold numbers: improved gaps and run times after tightening the relaxation; ``ns.": not solved to optimality tolerance within time limit; ``tl.": hits the time limit).}
\label{tab: ots}
{\footnotesize
\renewcommand{\arraystretch}{1.1}
\begin{tabular}{|rrrrrrrrrrrr|}
\hline
\multicolumn{1}{|l|}{}                        & \multicolumn{1}{l|}{}          & \multicolumn{5}{c|}{Optimality Gap (\%)}                                                                                                                                     & \multicolumn{5}{c|}{Runtime (seconds)}                                                                                                                \\
\multicolumn{1}{|r|}{Test Case}               & \multicolumn{1}{c|}{UB}        & \multicolumn{1}{c|}{{\cblue P}}   & \multicolumn{1}{c|}{E}           & \multicolumn{1}{c|}{EC}          & \multicolumn{1}{c|}{ECB}     & \multicolumn{1}{c|}{ECB*}      & \multicolumn{1}{c|}{{\cblue P}}      & \multicolumn{1}{c|}{E}     & \multicolumn{1}{c|}{EC}    & \multicolumn{1}{c|}{ECB}  & \multicolumn{1}{c|}{ECB*} \\ \hline
\multicolumn{12}{|c|}{Typical {\cblue Operating} Conditions (TYP)}    \\ \hline
\multicolumn{1}{|r|}{case3\_lmbd}             & \multicolumn{1}{r|}{5812.6}    & \multicolumn{1}{r|}{1.3}  & \multicolumn{1}{r|}{\textbf{1.0}}  & \multicolumn{1}{r|}{1.0}    & \multicolumn{1}{r|}{\textbf{0.0}} & \multicolumn{1}{r|}{0.1} & \multicolumn{1}{r|}{0.02}    & \multicolumn{1}{r|}{0.04}    & \multicolumn{1}{r|}{0.19}    & \multicolumn{1}{r|}{0.15}  & 0.39                      \\ \hline
\multicolumn{1}{|r|}{case5\_pjm}              & \multicolumn{1}{r|}{15174.0}   & \multicolumn{1}{r|}{1.1}  & \multicolumn{1}{r|}{1.1}           & \multicolumn{1}{r|}{1.1}           & \multicolumn{1}{r|}{1.1}   & \multicolumn{1}{r|}{1.1}        & \multicolumn{1}{r|}{0.12}    & \multicolumn{1}{r|}{0.04}    & \multicolumn{1}{r|}{0.23}    & \multicolumn{1}{r|}{0.31}  & 0.43                      \\ \hline
\multicolumn{1}{|r|}{case14\_ieee}            & \multicolumn{1}{r|}{2178.1}    & \multicolumn{1}{r|}{0.1}  & \multicolumn{1}{r|}{0.1}           & \multicolumn{1}{r|}{0.1}               & \multicolumn{1}{r|}{0.1}     & \multicolumn{1}{r|}{0.1}     & \multicolumn{1}{r|}{0.44}    & \multicolumn{1}{r|}{0.28}    & \multicolumn{1}{r|}{1.49}  & \multicolumn{1}{r|}{2.50}  & \textbf{1.32}                      \\ \hline
\multicolumn{1}{|r|}{case24\_ieee\_rts}       & \multicolumn{1}{r|}{63352.2}   & \multicolumn{1}{r|}{0.0}  & \multicolumn{1}{r|}{0.0}           & \multicolumn{1}{r|}{0.0}           & \multicolumn{1}{r|}{0.0}      & \multicolumn{1}{r|}{0.0}     & \multicolumn{1}{r|}{4.28}    & \multicolumn{1}{r|}{1.11}    & \multicolumn{1}{r|}{310.56} & \multicolumn{1}{r|}{349.21}   & \textbf{3.37}                    \\ \hline
\multicolumn{1}{|r|}{case30\_as}              & \multicolumn{1}{r|}{803.1}     & \multicolumn{1}{r|}{0.1}  & \multicolumn{1}{r|}{0.1}           & \multicolumn{1}{r|}{0.1}         & \multicolumn{1}{r|}{0.1}    & \multicolumn{1}{r|}{0.1}      & \multicolumn{1}{r|}{6.21}    & \multicolumn{1}{r|}{21.97}   & \multicolumn{1}{r|}{446.96}   & \multicolumn{1}{r|}{630.60} & \textbf{9.97}                    \\ \hline
\multicolumn{1}{|r|}{case30\_ieee}            & \multicolumn{1}{r|}{7579.0}    & \multicolumn{1}{r|}{12.1} & \multicolumn{1}{r|}{\textbf{11.9}} & \multicolumn{1}{r|}{11.9}      & \multicolumn{1}{r|}{\textbf{11.0}}     & \multicolumn{1}{r|}{11.0}     & \multicolumn{1}{r|}{1.21}    & \multicolumn{1}{r|}{1.01}    & \multicolumn{1}{r|}{3.30}  & \multicolumn{1}{r|}{6.56}  & \textbf{4.79}                      \\ \hline
\multicolumn{1}{|r|}{case39\_epri}            & \multicolumn{1}{r|}{137728.7}  & \multicolumn{1}{r|}{0.0}  & \multicolumn{1}{r|}{0.0}           & \multicolumn{1}{r|}{0.0}        & \multicolumn{1}{r|}{0.0}     & \multicolumn{1}{r|}{0.0}     & \multicolumn{1}{r|}{0.65}    & \multicolumn{1}{r|}{0.52}    & \multicolumn{1}{r|}{0.93}    & \multicolumn{1}{r|}{1.17}  & 2.32                      \\ \hline
\multicolumn{1}{|r|}{case57\_ieee}            & \multicolumn{1}{r|}{37559.3}   & \multicolumn{1}{r|}{0.1}  & \multicolumn{1}{r|}{0.1}           & \multicolumn{1}{r|}{0.1}              & \multicolumn{1}{r|}{0.1}     & \multicolumn{1}{r|}{0.1}      & \multicolumn{1}{r|}{34.34}   & \multicolumn{1}{r|}{17.51}   & \multicolumn{1}{r|}{40.27}   & \multicolumn{1}{r|}{77.43} & \textbf{27.96}                    \\ \hline
\multicolumn{1}{|r|}{case73\_ieee\_rts}       & \multicolumn{1}{r|}{189764.1}  & \multicolumn{1}{r|}{0.0}  & \multicolumn{1}{r|}{0.0}           & \multicolumn{1}{r|}{0.0}           & \multicolumn{1}{r|}{0.0}   & \multicolumn{1}{r|}{0.0}        & \multicolumn{1}{r|}{39.58}   & \multicolumn{1}{r|}{33.14}   & \multicolumn{1}{r|}{4876.16}  & \multicolumn{1}{r|}{5136.27} & \textbf{34.48}                   \\ \hline
\multicolumn{1}{|r|}{case89\_pegase}          & \multicolumn{1}{r|}{106622.2}  & \multicolumn{1}{r|}{0.1}  & \multicolumn{1}{r|}{0.1}           & \multicolumn{1}{r|}{\cblue ns.}         & \multicolumn{1}{r|}{ns.}   & \multicolumn{1}{r|}{ns.}         & \multicolumn{1}{r|}{tl.}     & \multicolumn{1}{r|}{tl.}     & \multicolumn{1}{r|}{tl.} & \multicolumn{1}{r|}{tl.}  & tl.                       \\ \hline
\multicolumn{1}{|r|}{case118\_ieee}           & \multicolumn{1}{r|}{96645.9}   & \multicolumn{1}{r|}{0.3}  & \multicolumn{1}{r|}{0.3}           & \multicolumn{1}{r|}{0.3}           & \multicolumn{1}{r|}{0.3}    & \multicolumn{1}{r|}{0.3}      & \multicolumn{1}{r|}{497.58}  & \multicolumn{1}{r|}{415.09}  & \multicolumn{1}{r|}{1622.20}  & \multicolumn{1}{r|}{1778.12} & \textbf{1389.84}                   \\ \hline
\multicolumn{1}{|r|}{case179\_goc}            & \multicolumn{1}{r|}{\cblue 754083.7}  & \multicolumn{1}{r|}{\cblue 0.1}  & \multicolumn{1}{r|}{\cblue 0.1}           & \multicolumn{1}{r|}{\cblue 0.1}           & \multicolumn{1}{r|}{\cblue 0.1}      & \multicolumn{1}{r|}{\cblue 0.1}     & \multicolumn{1}{r|}{1095.93} & \multicolumn{1}{r|}{307.07}  & \multicolumn{1}{r|}{175.25}  & \multicolumn{1}{r|}{334.33 }	& 340.78                    \\ \hline
\multicolumn{1}{|r|}{case200\_activ}          & \multicolumn{1}{r|}{27557.6}   & \multicolumn{1}{r|}{0.0}  & \multicolumn{1}{r|}{0.0}           & \multicolumn{1}{r|}{0.0}           & \multicolumn{1}{r|}{0.0}    & \multicolumn{1}{r|}{ns.}      & \multicolumn{1}{r|}{1636.15} & \multicolumn{1}{r|}{2837.87} & \multicolumn{1}{r|}{4398.64}  & \multicolumn{1}{r|}{3014.30}	& tl.                  \\ \hline
\multicolumn{12}{|c|}{Small Angle Difference Conditions (SAD)}                                                                                                                                                                                                                                                                                                                                                        \\ \hline
\multicolumn{1}{|r|}{case3\_lmbd\_sad}        & \multicolumn{1}{r|}{5959.3}    & \multicolumn{1}{r|}{3.0}  & \multicolumn{1}{r|}{\textbf{1.4}}  & \multicolumn{1}{r|}{\textbf{1.3}} & \multicolumn{1}{r|}{\textbf{0.1}}  & \multicolumn{1}{r|}{0.1} & \multicolumn{1}{r|}{0.02}    & \multicolumn{1}{r|}{0.02}    & \multicolumn{1}{r|}{0.07} & \multicolumn{1}{r|}{0.15}  & 0.38                      \\ \hline
\multicolumn{1}{|r|}{case5\_pjm\_sad}         & \multicolumn{1}{r|}{26108.8}   & \multicolumn{1}{r|}{1.4}  & \multicolumn{1}{r|}{\textbf{0.6}}  & \multicolumn{1}{r|}{0.6}         & \multicolumn{1}{r|}{\textbf{0.2}} & \multicolumn{1}{r|}{0.2} & \multicolumn{1}{r|}{0.04}    & \multicolumn{1}{r|}{0.05}    & \multicolumn{1}{r|}{0.15}    & \multicolumn{1}{r|}{0.20}  & 0.41                      \\ \hline
\multicolumn{1}{|r|}{case14\_ieee\_sad}       & \multicolumn{1}{r|}{2727.5}    & \multicolumn{1}{r|}{20.1} & \multicolumn{1}{r|}{\textbf{18.3}} & \multicolumn{1}{r|}{\textbf{12.1}} & \multicolumn{1}{r|}{\textbf{0.7}} & \multicolumn{1}{r|}{0.8} & \multicolumn{1}{r|}{0.41}    & \multicolumn{1}{r|}{0.57}    & \multicolumn{1}{r|}{2.55}  & \multicolumn{1}{r|}{3.23}   & \textbf{1.15}                      \\ \hline
\multicolumn{1}{|r|}{case24\_ieee\_rts\_sad}  & \multicolumn{1}{r|}{75794.0}   & \multicolumn{1}{r|}{5.3}  & \multicolumn{1}{r|}{\textbf{2.4}}  & \multicolumn{1}{r|}{\textbf{2.1}}  & \multicolumn{1}{r|}{\textbf{0.7}}  & \multicolumn{1}{r|}{0.8} & \multicolumn{1}{r|}{32.92}   & \multicolumn{1}{r|}{14.37}   & \multicolumn{1}{r|}{84.41}   & \multicolumn{1}{r|}{104.09}  & \textbf{16.97}                    \\ \hline
\multicolumn{1}{|r|}{case30\_as\_sad}         & \multicolumn{1}{r|}{893.9}     & \multicolumn{1}{r|}{4.5}  & \multicolumn{1}{r|}{\textbf{1.9}}  & \multicolumn{1}{r|}{1.9}    & \multicolumn{1}{r|}{\textbf{1.1}}     & \multicolumn{1}{r|}{1.2}     & \multicolumn{1}{r|}{18.56}   & \multicolumn{1}{r|}{11.68}   & \multicolumn{1}{r|}{45.37}   & \multicolumn{1}{r|}{65.27} & \textbf{14.28}                     \\ \hline
\multicolumn{1}{|r|}{case30\_ieee\_sad}       & \multicolumn{1}{r|}{8188.6}    & \multicolumn{1}{r|}{8.8}  & \multicolumn{1}{r|}{\textbf{8.7}}  & \multicolumn{1}{r|}{8.7}  & \multicolumn{1}{r|}{\textbf{0.1}} & \multicolumn{1}{r|}{0.2} & \multicolumn{1}{r|}{1.50}    & \multicolumn{1}{r|}{2.82}    & \multicolumn{1}{r|}{5.62}    & \multicolumn{1}{r|}{5.89}  & \textbf{2.77}                      \\ \hline
\multicolumn{1}{|r|}{case39\_epri\_sad}       & \multicolumn{1}{r|}{147472.8}  & \multicolumn{1}{r|}{0.1}  & \multicolumn{1}{r|}{0.1}           & \multicolumn{1}{r|}{0.1}        & \multicolumn{1}{r|}{0.1}     & \multicolumn{1}{r|}{0.1}      & \multicolumn{1}{r|}{11.84}   & \multicolumn{1}{r|}{15.59}   & \multicolumn{1}{r|}{14.39}  & \multicolumn{1}{r|}{16.40}  & \textbf{9.68}                     \\ \hline
\multicolumn{1}{|r|}{case57\_ieee\_sad}       & \multicolumn{1}{r|}{38597.8}   & \multicolumn{1}{r|}{0.2}  & \multicolumn{1}{r|}{0.2}           & \multicolumn{1}{r|}{\textbf{0.1}}  & \multicolumn{1}{r|}{0.1}     & \multicolumn{1}{r|}{0.1}     & \multicolumn{1}{r|}{24.21}   & \multicolumn{1}{r|}{44.55}   & \multicolumn{1}{r|}{148.33}   & \multicolumn{1}{r|}{189.35} & \textbf{82.79}                    \\ \hline
\multicolumn{1}{|r|}{case89\_pegase\_sad}     & \multicolumn{1}{r|}{\cblue 106633.6}  & \multicolumn{1}{r|}{ns.}    & \multicolumn{1}{r|}{\cblue \textbf{0.1}}  & \multicolumn{1}{r|}{\cblue ns.}         & \multicolumn{1}{r|}{ns.}    & \multicolumn{1}{r|}{ns.}         & \multicolumn{1}{r|}{tl.}     & \multicolumn{1}{r|}{tl.}     & \multicolumn{1}{r|}{tl.}   & \multicolumn{1}{r|}{tl.}  & tl.                       \\ \hline
\multicolumn{1}{|r|}{case118\_ieee\_sad}      & \multicolumn{1}{r|}{97572.5}   & \multicolumn{1}{r|}{0.9}  & \multicolumn{1}{r|}{0.9}           & \multicolumn{1}{r|}{0.9}        & \multicolumn{1}{r|}{0.9}      & \multicolumn{1}{r|}{0.9}     & \multicolumn{1}{r|}{4581.28} & \multicolumn{1}{r|}{3453.41} & \multicolumn{1}{r|}{6818.73}   & \multicolumn{1}{r|}{tl.}  & \textbf{4520.33}                       \\ \hline
\multicolumn{1}{|r|}{case179\_goc\_sad}       & \multicolumn{1}{r|}{755293.1}  & \multicolumn{1}{r|}{ns.}    & \multicolumn{1}{r|}{\textbf{0.1}}  & \multicolumn{1}{r|}{0.1}           & \multicolumn{1}{r|}{0.1}      & \multicolumn{1}{r|}{0.1}    & \multicolumn{1}{r|}{tl.}     & \multicolumn{1}{r|}{tl.}     & \multicolumn{1}{r|}{tl.}     & \multicolumn{1}{r|}{tl.}  & tl.                       \\ \hline
\multicolumn{1}{|r|}{case200\_activ\_sad}     & \multicolumn{1}{r|}{27557.6}   & \multicolumn{1}{r|}{0.0}  & \multicolumn{1}{r|}{0.0}           & \multicolumn{1}{r|}{0.0}           & \multicolumn{1}{r|}{0.0}     & \multicolumn{1}{r|}{ns.}     & \multicolumn{1}{r|}{3564.09} & \multicolumn{1}{r|}{tl.}     & \multicolumn{1}{r|}{1671.87}  & \multicolumn{1}{r|}{tl.}		& tl.                       \\ \hline
\multicolumn{12}{|c|}{Congested Operating Conditions (API)}\\ \hline
\multicolumn{1}{|r|}{case3\_lmbd\_api}        & \multicolumn{1}{r|}{10636.0}   & \multicolumn{1}{r|}{3.8}  & \multicolumn{1}{r|}{\textbf{0.4}}  & \multicolumn{1}{r|}{0.4}      & \multicolumn{1}{r|}{\textbf{0.0}}    & \multicolumn{1}{r|}{0.0}       & \multicolumn{1}{r|}{0.02}    & \multicolumn{1}{r|}{0.03}    & \multicolumn{1}{r|}{0.13}  & \multicolumn{1}{r|}{0.09}  & 0.33                      \\ \hline
\multicolumn{1}{|r|}{case5\_pjm\_api}         & \multicolumn{1}{r|}{75190.3}   & \multicolumn{1}{r|}{2.6}  & \multicolumn{1}{r|}{2.6}           & \multicolumn{1}{r|}{2.6}    & \multicolumn{1}{r|}{\textbf{0.3}} & \multicolumn{1}{r|}{0.3} & \multicolumn{1}{r|}{0.11}    & \multicolumn{1}{r|}{0.07}    & \multicolumn{1}{r|}{0.19}    & \multicolumn{1}{r|}{0.30}  & 0.55                      \\ \hline
\multicolumn{1}{|r|}{case14\_ieee\_api}       & \multicolumn{1}{r|}{5999.4}    & \multicolumn{1}{r|}{5.1}  & \multicolumn{1}{r|}{5.1}           & \multicolumn{1}{r|}{5.1}      & \multicolumn{1}{r|}{\textbf{0.8}} & \multicolumn{1}{r|}{0.9} & \multicolumn{1}{r|}{0.34}    & \multicolumn{1}{r|}{0.27}    & \multicolumn{1}{r|}{1.24}     & \multicolumn{1}{r|}{0.82}  & 0.84                      \\ \hline
\multicolumn{1}{|r|}{case24\_ieee\_rts\_api}  & \multicolumn{1}{r|}{119743.1}  & \multicolumn{1}{r|}{5.2}  & \multicolumn{1}{r|}{\textbf{3.4}}  & \multicolumn{1}{r|}{3.4}      & \multicolumn{1}{r|}{\textbf{1.2}}     & \multicolumn{1}{r|}{1.2}     & \multicolumn{1}{r|}{7.30}    & \multicolumn{1}{r|}{4.16}    & \multicolumn{1}{r|}{82.80}    & \multicolumn{1}{r|}{31.66}  & \textbf{6.50}                     \\ \hline
\multicolumn{1}{|r|}{case30\_as\_api}         & \multicolumn{1}{r|}{\cblue 2925.1}    & \multicolumn{1}{r|}{\cblue 5.4}  & \multicolumn{1}{r|}{\cblue 5.4}           & \multicolumn{1}{r|}{\cblue 5.4}      & \multicolumn{1}{r|}{\cblue \textbf{5.1}} & \multicolumn{1}{r|}{\cblue 5.1}	& \multicolumn{1}{r|}{3.65}    & \multicolumn{1}{r|}{3.44}    & \multicolumn{1}{r|}{37.99}   & \multicolumn{1}{r|}{20.32}   & \textbf{6.00}                     \\ \hline
\multicolumn{1}{|r|}{case30\_ieee\_api}       & \multicolumn{1}{r|}{17936.5}   & \multicolumn{1}{r|}{4.9}  & \multicolumn{1}{r|}{4.9}           & \multicolumn{1}{r|}{4.9}   & \multicolumn{1}{r|}{\textbf{0.3}} & \multicolumn{1}{r|}{0.4}	& \multicolumn{1}{r|}{1.25}    & \multicolumn{1}{r|}{0.86}    & \multicolumn{1}{r|}{2.22}   & \multicolumn{1}{r|}{4.57}  & \textbf{2.22}                      \\ \hline
\multicolumn{1}{|r|}{case39\_epri\_api}       & \multicolumn{1}{r|}{246723.0}  & \multicolumn{1}{r|}{0.5}  & \multicolumn{1}{r|}{0.5}           & \multicolumn{1}{r|}{0.5}    & \multicolumn{1}{r|}{\textbf{0.4}}     & \multicolumn{1}{r|}{0.4}     & \multicolumn{1}{r|}{1.42}    & \multicolumn{1}{r|}{0.75}    & \multicolumn{1}{r|}{1.86}  & \multicolumn{1}{r|}{4.84} & \textbf{3.77}                      \\ \hline
\multicolumn{1}{|r|}{case57\_ieee\_api}       & \multicolumn{1}{r|}{49271.9}   & \multicolumn{1}{r|}{0.0}  & \multicolumn{1}{r|}{0.1}           & \multicolumn{1}{r|}{0.1}    & \multicolumn{1}{r|}{\textbf{0.0}}   & \multicolumn{1}{r|}{0.0}       & \multicolumn{1}{r|}{36.50}   & \multicolumn{1}{r|}{12.63}   & \multicolumn{1}{r|}{44.12}   & \multicolumn{1}{r|}{51.50} & \textbf{27.34}                     \\ \hline
\multicolumn{1}{|r|}{case73\_ieee\_rts\_api}  & \multicolumn{1}{r|}{385277.3}  & \multicolumn{1}{r|}{4.3}  & \multicolumn{1}{r|}{\textbf{2.4}}  & \multicolumn{1}{r|}{\cblue ns.}   & \multicolumn{1}{r|}{\textbf{1.2}}   & \multicolumn{1}{r|}{1.2}       & \multicolumn{1}{r|}{128.08}  & \multicolumn{1}{r|}{3869.93} & \multicolumn{1}{r|}{tl.} & \multicolumn{1}{r|}{3307.26}	& \textbf{328.03}                   \\ \hline
\multicolumn{1}{|r|}{case89\_pegase\_api}     & \multicolumn{1}{r|}{100325.3}  & \multicolumn{1}{r|}{ns.}    & \multicolumn{1}{r|}{\textbf{0.2}}  & \multicolumn{1}{r|}{\cblue ns.}      & \multicolumn{1}{r|}{ns.}     & \multicolumn{1}{r|}{ns.}       & \multicolumn{1}{r|}{tl.}     & \multicolumn{1}{r|}{tl.}     & \multicolumn{1}{r|}{tl.}     & \multicolumn{1}{r|}{tl.}  & tl.                       \\ \hline
\multicolumn{1}{|r|}{case118\_ieee\_api}      & \multicolumn{1}{r|}{181535.8}  & \multicolumn{1}{r|}{6.5}  & \multicolumn{1}{r|}{\textbf{6.2}}  & \multicolumn{1}{r|}{6.2}    & \multicolumn{1}{r|}{\textbf{6.1}}      & \multicolumn{1}{r|}{6.1}    & \multicolumn{1}{r|}{tl.}     & \multicolumn{1}{r|}{tl.}     & \multicolumn{1}{r|}{tl.}   & \multicolumn{1}{r|}{tl.}	& tl.                       \\ \hline
\multicolumn{1}{|r|}{case162\_ieee\_dtc\_api} & \multicolumn{1}{r|}{116923.8}  & \multicolumn{1}{r|}{1.0}  & \multicolumn{1}{r|}{1.0}           & \multicolumn{1}{r|}{\cblue ns.}       & \multicolumn{1}{r|}{\cblue ns.} & \multicolumn{1}{r|}{ns.}	& \multicolumn{1}{r|}{tl.}     & \multicolumn{1}{r|}{tl.}     & \multicolumn{1}{r|}{tl.}   & \multicolumn{1}{r|}{tl.}   & tl.                       \\ \hline
\multicolumn{1}{|r|}{case179\_goc\_api}       & \multicolumn{1}{r|}{\cblue 1932020.5} & \multicolumn{1}{r|}{6.3}  & \multicolumn{1}{r|}{\textbf{5.9}}  & \multicolumn{1}{r|}{\textbf{5.8}}    & \multicolumn{1}{r|}{\textbf{0.6}}      & \multicolumn{1}{r|}{0.6}     & \multicolumn{1}{r|}{tl.}     & \multicolumn{1}{r|}{881.61}  & \multicolumn{1}{r|}{tl.}  & \multicolumn{1}{r|}{343.73} & 1000.46                    \\ \hline
\multicolumn{1}{|r|}{case200\_activ\_api}     & \multicolumn{1}{r|}{35701.3}   & \multicolumn{1}{r|}{0.0}  & \multicolumn{1}{r|}{0.0}           & \multicolumn{1}{r|}{0.0}         & \multicolumn{1}{r|}{0.0}      & \multicolumn{1}{r|}{0.0}    & \multicolumn{1}{r|}{1993.90} & \multicolumn{1}{r|}{2764.12} & \multicolumn{1}{r|}{tl.}   & \multicolumn{1}{r|}{3029.17} & \textbf{855.60}                   \\ \hline
\multicolumn{1}{|r|}{case240\_pserc\_api}     & \multicolumn{1}{r|}{\cblue 4638308.7} & \multicolumn{1}{r|}{ns.}    & \multicolumn{1}{r|}{\textbf{0.6}}  & \multicolumn{1}{r|}{0.6}             & \multicolumn{1}{r|}{\cblue \textbf{0.5}}      & \multicolumn{1}{r|}{0.6}     & \multicolumn{1}{r|}{tl.}     & \multicolumn{1}{r|}{tl.}     & \multicolumn{1}{r|}{tl.}     & \multicolumn{1}{r|}{tl.} & tl.                       \\ \hline
\multicolumn{1}{|r|}{case300\_ieee\_api}      & \multicolumn{1}{r|}{684985.5}  & \multicolumn{1}{r|}{0.8}  & \multicolumn{1}{r|}{0.8}           & \multicolumn{1}{r|}{0.8}   & \multicolumn{1}{r|}{\textbf{0.7}}     & \multicolumn{1}{r|}{0.8}     & \multicolumn{1}{r|}{tl.}     & \multicolumn{1}{r|}{tl.}     & \multicolumn{1}{r|}{tl.}     & \multicolumn{1}{r|}{tl.}  & tl.                       \\ \hline
\end{tabular}}
\end{table*}

In Table \ref{tab: ots}, compared with {\cblue``P"}, our tightened ``E" reduces the optimality gap for many benchmark instances, especially for the SAD and API 
ones. For example, it yields 3.4\% gap improvement 
for case3\_lmbd\_api and 2.9\% improvement for case24\_ieee\_rts\_sad. It also solves several instances to optimality that {\cblue``P"} is not able to solve within the time limit. The benefit of the \cycconstrs \ (``EC") is most apparent for ``SAD" , with case14\_ieee\_sad closing 6.2\%.

Combining the extreme point formulation, OBBT algorithm, and the \cycconstrs, \textit{we obtain the tightest {\cred QC-based} ACOTS relaxation in the literature}, as highlighted in the optimality gap columns of ``ECB" and ``ECB*" (see Table \ref{tab: ots}). Note that {\cblue ``ECB*" may not be as tight as ``ECB," since ``ECB" includes all cycle constraints, while ``ECB*" adds up to 200 cycle constraints via cutting planes. However, when ``ECB*" is not as tight, the difference is only 0.1\%, indicating that ``ECB*" remains relatively strong. Due to the efficient implementation of the branch-and-cut framework, ``ECB*" }reduces the solution time significantly in many instances. It is clear from the table that the OBBT algorithm in conjunction with all the proposed enhancements in this paper can provide significant improvements in closing the gap for several benchmark cases. For example, in case14\_ieee\_sad, ``ECB" closes as much as 17.6\% of the gap when compared with ``E", and proves global optimality for case3\_lmbd\_api. {\cblue Compared with ``E", ``ECB" and the faster ``ECB*" reduce the gap of nine instances to below 1.0\%. In fact, with ``ECB" and ``ECB*", }we close the optimality gaps to {\cblue less} than 1.0\% for $\approx$75\% of all instances; these improvements are also significant when compared with state-of-the-art implementation in {\cblue ``P''} and the results in \citep{bestuzheva2020convex}. {\cred Note that if the optimality gap equals zero, then the corresponding ACOTS relaxation provides a globally optimal solution to the non-convex ACOTS problem.} 

{\cblue Note that there are a few instances that are not solved due to no feasible solution found within the time limit. In particular, for the instances with 89 buses, we do not provide warm-start solutions when cycle constraints are included, as these instances have a large number of cycles, resulting in large warm-start files that cause a stack overflow error in the solver. In addition, the warm-start solution is incorrectly rejected by the solver when solving case200\_activ with ``ECB*". Without warm start, these instances struggle to find a feasible solution.}

{\cblue 
We also check the tightness of the root node relaxation. To find the values of the root node relaxation, we set the Gurobi ``NodeLimit" parameter to 0 and obtain the optimal objective bound. In Table \ref{tab: root_relax} of the Online Supplement, we present the optimality gaps at the root node for the ACOTS relaxations, which equals 100$\times$(UB-root node relaxation value)/UB. In contrast to the results in Table \ref{tab: ots}, the improvement of ``E" is relatively small at the root node, indicating that ``E" closes more gap during branch-and-bound. ``EC" provides tight root relaxations for several SAD instances. The ``ECB" relaxation provides tighter root node relaxations for many instances, in particular for SAD and API ones, which possibly explains the reduced final optimality gaps in Table \ref{tab: ots}. Note that unlike ``ECB", the bound tightening optimization models for ``ECB*" are not strengthened by cycle constraints, leading to potentially weaker bounds after OBBT and weaker root node relaxations. In fact, ``ECB*" shows weaker root relaxations than ``ECB" in several instances, suggesting that the tighter root relaxation of ``ECB" is a result of both bound tightening and cycle constraints.

In addition, we include the results for the ``PB" and ``EB" relaxations (i.e., tightened ``P" and ``E" respectively with OBBT) in Table \ref{tab: eb_appendix} of the Online Supplement. We highlight with boldface-italic the optimality gaps that can be further improved by ``EB" (for ``PB") and by ``ECB" (for ``EB"). The difference between ``PB" and ``EB" shows that our extreme-point representation is useful for tightening many instances, even after OBBT is implemented. Also, adding lifted cycle constraints (``ECB") tightens ``EB" for many instances.

To summarize, our basic relaxation with extreme-point representation (i.e., ``E") provides improvements over many of the instances compared with the state-of-the-art benchmark, with comparable computational time. Both the \cycconstrs~and OBBT provide enhancements to the basic relaxation. Combining all the improvements provides the tightest QC-based ACOTS relaxation in the literature, with the runtime comparable to the basic relaxation when implemented in the branch-and-cut format.}

{\cblue Note that for obtaining the UB, Juniper fails to find feasible solutions for many instances within the time limit (2 hours), while Knitro {\dblue found feasible solutions for} all instances. On the other hand, for the instances that Juniper can solve, it sometimes provides better solutions. We provide a comparison of the results from the two solvers in Table \ref{tab: junper_knitro} of the Online Supplement.}

Although not shown in the result table, it is worthy to mention that tightening the ACOTS relaxations does lead to different line switching decisions. {\cblue For example, in ``case24\_ieee\_rts\_api", 13.9\% of the lines have different switching decisions between ``E" and ``ECB". If we fix the on/off statuses of the lines based on the solution of ``E" and then solve the problem with ``ECB", the cost increases by 4.5\%. Similarly, in ``case30\_as\_api", 7.3\% of the lines have different switching decisions, leading to a 79.0\% cost increase when using the ``E" solution. }Therefore, by tightening the relaxation, we make better decisions and obtain better approximations of the true cost after line switching. 

{\cblue 
\subsection{Results with the Maximum Spanning Tree Heuristic}\label{sec: experiment_heuristic}
Table \ref{tab: span_tree} compares results of our ACOTS relaxations with and without the maximum spanning tree heuristic proposed in Section \ref{sec: span_tree}. We include PGLib instances with up to 300 buses, which are solved with the same setups as in Section \ref{sec: experiment_acots}. In addition, when using the heuristic, we solve one ``ECB" instance (``case30\_as\_api") and the three 200-bus ``ECB*" instances with Gurobi's presolve disabled and with one thread. This is necessary to avoid the solver either incorrectly terminating with a locally optimal solution or incorrectly reporting infeasibility.

\begin{table}[htbp]
\centering
    \caption{\cblue Comparing the objective difference and saved runtime of ACOTS relaxations with and without the spanning tree heuristic (``s.": solved with heuristic and unsolved without; ``ns.": unsolved with either method; ``ws." solved without the heuristic and unsolved with it).}\label{tab: span_tree}
{\cblue \footnotesize
    \begin{tabular}{|r|>{\raggedleft\arraybackslash}p{0.8cm}|>{\raggedleft\arraybackslash}p{0.8cm}|r|r|r|r|r|r|}
\hline
\multicolumn{1}{|l|}{}          & \multicolumn{4}{c|}{Objective Difference (\%)}  & \multicolumn{4}{c|}{Runtime Saved (seconds)} \\
\multicolumn{1}{|r|}{Test Case}  & \multicolumn{1}{c|}{E}           & \multicolumn{1}{c|}{EC}          & \multicolumn{1}{c|}{ECB}     & \multicolumn{1}{c|}{ECB*}    & \multicolumn{1}{c|}{E}     & \multicolumn{1}{c|}{EC}    & \multicolumn{1}{c|}{ECB}  & \multicolumn{1}{c|}{ECB*} \\ \hline
\multicolumn{9}{|c|}{Typical Operating Conditions (TYP)} \\
    \hline
    case3\_lmbd & 0.0   & 0.0   & 0.0   & 0.0   & 0.0   & 0.1  & 0.0   & 0.2 \\
    \hline
    case5\_pjm & 0.0   & 0.0   & 0.1   & 0.1   & 0.0   & 0.1   & 0.1   & -0.9 \\
    \hline
    case14\_ieee & 0.0   & 0.0   & 0.0   & 0.0   & 0.0   & 0.1   & -0.2  & 2.4 \\
    \hline
    case24\_ieee\_rts & 0.0   & 0.0   & 0.0   & 0.0   & -0.4  & 304.2 & 339.9 & 2.7 \\
    \hline
    case30\_as & 0.0   & 0.0   & 0.0   & 0.0   & 18.9  & 433.1 & 593.7 & -15.7 \\
    \hline
    case30\_ieee & 0.0   & 0.0   & 20.2  & 20.2  & 0.4   & 1.2   & -2.1  & 6.7 \\
    \hline
    case39\_epri & 0.0   & 0.0   & 0.0   & 0.1   & 0.1   & 0.2   & -1.2  & 3.3 \\
    \hline
    case57\_ieee & 0.0   & 0.0   & 0.0   & 0.0   & 14.4  & 31.6  & 56.9  & 17.2 \\
    \hline
    case73\_ieee\_rts & 0.0   & 0.0   & 0.0   & 0.0   & 13.6  & 4796.2 & 5037.2 & -1.9 \\
    \hline
    case89\_pegase & 0.0   & ns.   & ns.   & ns.   & 4667.2 & -     &  -     & - \\
    \hline
    case118\_ieee & 0.0   & 0.1   & 0.2   & 0.1   & 364.1 & 1251.9 & -334.1 & 1200.9 \\
    \hline
    case162\_ieee\_dtc & \textbf{s.} & ns.   & ns.   & ns.   & -     & -     & -     & - \\
    \hline
    case179\_goc & 0.0   & 0.0   & 0.0   & 0.0   & 296.8 & 129.0 & 2.2   & -220.0 \\
    \hline
    case200\_activ & 0.0   & 0.0   & 0.0   & \textbf{s.} & 2654.5 & 4307.8 & -991.4 & - \\
    \hline
    case240\_pserc & 0.0   & 0.0   & ns.   & 0.2   & 4653.3 & 1778.3 & -     & 0.0 \\
    \hline
    case300\_ieee & ns.   & ns.   & ns.   & ns.   & -     & -     & -     & - \\
    \hline
    \multicolumn{9}{|c|}{Small Angle Difference Conditions (SAD)} \\
    \hline
    case3\_lmbd\_sad & 0.0   & 0.0   & 0.0   & 0.0   & 0.0   & 0.0  & 0.0   & 0.3 \\
    \hline
    case5\_pjm\_sad & 0.0   & 0.0   & 0.1   & 0.1   & 0.0   & 0.0   & 0.0   & -0.1 \\
    \hline
    case14\_ieee\_sad & 0.4   & 0.4   & 0.5   & 0.6   & 0.3   & 0.5   & 1.7   & 0.0 \\
    \hline
    case24\_ieee\_rts\_sad & 1.1   & 1.3   & 1.9   & 2.0   & 13.4  & 78.8  & 94.8  & 14.6 \\
    \hline
    case30\_as\_sad & 0.0   & 0.0   & 1.1   & 1.1   & 7.9   & 27.7  & 30.1  & 6.3 \\
    \hline
    case30\_ieee\_sad & 3.5   & 3.6   & 0.1   & 0.1   & 1.9   & 1.4   & -1.4  & -0.3 \\
    \hline
    case39\_epri\_sad & 0.4   & 0.4   & 0.5   & 0.5   & 14.8  & 12.9  & 13.4  & 7.6 \\
    \hline
    case57\_ieee\_sad & 0.0   & 0.0   & 0.1   & 0.1   & 39.9  & 121.0 & 135.7 & 67.3 \\
    \hline
    case73\_ieee\_rts\_sad & \textbf{s.} & \textbf{s.} & \textbf{s.} & \textbf{s.} & -     & -     & -     & - \\
    \hline
    case89\_pegase\_sad & 0.0   & ns.   & ns.   & ns.   & 0.0   & -     &  -     & - \\
    \hline
    case118\_ieee\_sad & 1.4   & 1.4   & ns.   & 1.7   & 6416.4 & -312.4 & -     & 1261.6 \\
    \hline
    case162\_ieee\_dtc\_sad & ns.   & ns.   & ns.   & ns.   & -     & -     & -     & - \\
    \hline
    case179\_goc\_sad & 0.1   & 0.1   & ws.   & 0.2   & 6885.0 & 5242.6 & -     & 5368.4 \\
    \hline
    case200\_activ\_sad & 0.0   & 0.0   & 0.0   & \textbf{s.} & 6980.2 & 6292.4 & 4481.7 & - \\
    \hline
    case240\_pserc\_sad & ns.   & ns.   & ns.   & ns.   & -     & -     & -     & - \\
    \hline
    case300\_ieee\_sad & ns.   & ns.   & ns.   & ns.   & -     & -     & -     & - \\
    \hline
    \multicolumn{9}{|c|}{Congested Operating Conditions (API)} \\
    \hline
    case3\_lmbd\_api & 1.2   & 2.0   & 5.6   & 5.6   & 0.0   & 0.1 & -0.1  & -1.1 \\
    \hline
    case5\_pjm\_api & 0.0   & 0.0   & 1.8   & 1.8   & 0.1   & 0.1   & 0.1   & 1.3 \\
    \hline
    case14\_ieee\_api & 0.0   & 0.0   & 0.3   & 0.4   & 0.1   & 0.3   & -0.8  & -0.4 \\
    \hline
    case24\_ieee\_rts\_api & 3.5   & 3.7   & 12.1  & 12.0  & 3.1   & 77.7  & 22.0  & 15.1 \\
    \hline
    case30\_as\_api & 0.0   & 0.0   & 29.3  & 29.0  & 1.8   & 27.2  & -27.0 & 2.6 \\
    \hline
    case30\_ieee\_api & 0.0   & 0.0   & 0.2   & 0.1   & 0.3   & 0.5   & -0.8  & 0.0 \\
    \hline
    case39\_epri\_api & 0.0   & 0.0   & 1.5   & 1.5   & 0.3   & 1.0   & 2.5   & 7.1 \\
    \hline
    case57\_ieee\_api & 0.0   & 0.0   & 0.0   & 0.0   & 9.9   & 36.1  & 32.5  & 18.7 \\
    \hline
    case73\_ieee\_rts\_api & 1.2   & \textbf{s.} & 7.1   & 7.2   & 3818.8 & -     & 3068.8 & 294.7 \\
    \hline
    case89\_pegase\_api & 0.0   & ns.   & ns.   & ns.   & 4415.6 & -     & -      & - \\
    \hline
    case118\_ieee\_api & 0.1   & 0.1   & ws.   & \textbf{s.} & 7113.7 & 6853.6 & -     & - \\
    \hline
    case162\_ieee\_dtc\_api & 0.0   & \textbf{s.} & ns.   & \textbf{s.} & 5543.4 & -     & -     & - \\
    \hline
    case179\_goc\_api & 0.0   & 0.1   & 0.0   & 0.0   & 729.2 & 7074.7 & -492.9 & 925.7 \\
    \hline
    case200\_activ\_api & 0.0   & 0.0   & 0.0   & 0.0   & 2404.3 & 6913.0 & 2016.3 & 647.8 \\
    \hline
    case240\_pserc\_api & 0.0   & 0.0   & \textbf{s.} & ws.   & 5559.7 & 0.0   & -     & - \\
    \hline
    case300\_ieee\_api & 0.0   & 0.0   & 0.1   & 0.2   & 6316.0 & 5198.2 & 0.0   & 5655.5 \\
    \hline
    \end{tabular}}
\end{table}

The ``Objective Difference" values are calculated by 100 $\times$(LB$_2$ - LB$_1$)/LB$_1$, where LB$_1$ and LB$_2$ are respectively the optimal value with and without the heuristic. We use ``s." for cases solved with the heuristic but not without it, ``ws." for cases solved without the heuristic but not with it, and ``ns." for cases not solved by either method. ``Runtime Saved" shows the saved runtime by using the heuristic. 

With the heuristic, we can solve many instances significantly faster, and solve 12 more instances within the 0.1\% tolerance. The heuristic can lead to higher optimal values for some instances. The objective difference is generally smaller for larger cases, suggesting that their accuracy is less affected by the heuristic. This is probably because larger networks tend to have more flexibility even after the heuristic fixes some lines to be switched on. There are three instances that can be solved without the heuristic, while the heuristic leads to relative gaps larger than 0.1\% (and below 0.5\%) at termination.

We also tested with assigning the same weight to all lines, but this approach does not lead to as much improvement as shown in Table \ref{tab: span_tree}. }

{\cblue 
\subsection{Performance on Large-Scale ACOTS Instances}\label{sec: large_inst}
In this section, we present experimental results for large-scale PGLib instances with 500 to 2,312 buses. Note that to the best of our knowledge, the largest ACOTS-QC instances investigated in the literature are up to 300 buses \citep{hijazi2017convex, bestuzheva2020convex}. However, with our proposed relaxations and heuristic, we are able to obtain results for some of these larger instances. 

In Table \ref{tab: large_inst_paper} we present results for the upper bound and the relaxations of ACOTS-QC. Here we include instances for which a feasible solution is found by either ``P", ``E", or ``EC", and provide the results for all instances in Table \ref{tab: large_inst} of the Online Supplement. The ``UB" and relaxation results are obtained with a similar method as in Section \ref{tab: ots}, with the following adjustment: solver time limits are set to 3 hours, and solver feasibility tolerances for the relaxations are set to $10^{-3}$ to reduce the likelihood of incorrectly rejecting warm-start solutions. Note that when using Gurobi's default feasibility tolerance of $10^{-6}$, none of the instances is solved. We include results for ``P", ``E", and ``EC" relaxations. None of the ``ECB" and ``ECB*" relaxation instances is solved within the time limit. The performance measures we use include the objective value and runtime (or the relative gap at termination for cases hit the time limit). We use ``nf." as the objective value if no feasible solution is found within the time limit.

{\cblue \footnotesize
    \begin{longtable}[htbp]{|r|r|r|r|r|>{\raggedleft\arraybackslash}p{1.4cm}|>{\raggedleft\arraybackslash}p{1.4cm}|>{\raggedleft\arraybackslash}p{1.4cm}|}
    \caption{\cblue \small \bf Objective value, runtime, and relative gap of ACOTS relaxations for instances with 500 to 2,312 buses (``nf.": no feasible solution within time limit).}\label{tab: large_inst_paper}\\
    \hline
          &       & \multicolumn{3}{c|}{Objective Value} & \multicolumn{3}{c|}{Runtime (seconds) (Relative Gap)} \\
    Test Case & \multicolumn{1}{c|}{UB} & \multicolumn{1}{c|}{P} & \multicolumn{1}{c|}{E} & \multicolumn{1}{c|}{EC} & \multicolumn{1}{c|}{P} & \multicolumn{1}{c|}{E} & \multicolumn{1}{c|}{EC} \\
    \hline
    case500\_goc\_sad & 487397.2 & nf.   & 455966.2 & nf.   & -     & (0.5\%) & - \\
    \hline
    case1354\_pegase\_sad & 1258848.0 & nf.   & 1239689.5 & 1239701.2 & -     & (0.1\%) & (0.2\%) \\
    \hline
    case2000\_goc\_sad & 992879.8 & nf.   & nf.   & 979501.4 & -     & -     & (65.1\%) \\
    \hhline{|=|=|=|=|=|=|=|=|}
    case500\_goc\_api & 675967.9 & nf.   & 668613.9 & 668617.9 & -     & 2332.84 & 9289.02 \\
    \hline
    case588\_sdet\_api & 391951.0 & 389971.1 & nf.   & nf.   & (0.2\%) & -     & - \\
    \hline
    case1354\_pegase\_api & 1498271.0 & nf.   & nf.   & 1490268.6 & -     & -     & (0.1\%) \\
    \hline
    case2000\_goc\_api & 1468630.3 & nf.   & nf.   & 1438342.5 & -     & -     & (86.2\%) \\
    \hline
    case2312\_goc\_api & 571446.1 & nf.   & 496817.5 & nf.   & -     & (26.3\%) & - \\
    \hhline{|=|=|=|=|=|=|=|=|}
    case500\_goc & 454946.0 & nf.   & 453852.0 & nf.   & -     & 8897.06 & - \\
    \hline
    case588\_sdet & 310016.0 & 307339.6 & nf.   & nf.   & (0.1\%) & -     & - \\
    \hline
    case793\_goc & 259097.8 & 256776.8 & nf.   & nf.   & (0.0\%) & -     & - \\
    \hline
    case1354\_pegase & 1258844.0 & nf.   & 1239314.7 & 1239327.2 & -     & 1770.48 & (0.1\%) \\
    \hline
    case1951\_rte & 2085530.3 & nf.   & 2083344.3 & nf.   & -     & (0.1\%) & - \\
    \hline
    case2000\_goc & 973432.5 & nf.   & 970466.8 & nf.   & -     & (57.8\%) & - \\
    \hline
    \end{longtable}}

The state-of-the-art ``P" relaxation solves 3 cases with at most 793 buses. In comparison, the ``E" relaxation solves more cases to a below 0.5\% relative gap, including some cases with over 1000 buses such as ``case1354\_pegase\_sad" and ``case1951\_rte". In particular, the objective value of ``case1951\_rte" is within 0.1\% of UB, suggesting it is very close to the global optimal solution. The ``EC" relaxation solves ``case1354\_pegase\_api" within a 0.1\% relative gap, while this case is not solved by ``E".

In Table \ref{tab: large_span_paper}, we also present results with the maximum spanning tree heuristic. Here we include instances that obtains a feasible solution from any one of the ACOTS-QC relaxations, and present the results for all instances in Table \ref{tab: large_span} of the Online Supplement. This method enables us to heuristically solve more large instances either to optimality or to within a 0.5\% relative gap. Interestingly, ``EC" can sometimes solve cases that ``E" cannot solve, such as ``case1354\_pegase\_api" and ``case2000\_goc\_api". This is likely due to the cycle constraints reduce the search space. The ``ECB" and ``ECB*" relaxations solve fewer instances compared with ``E" and ``EC", which is possibly because more warm-start solutions being incorrectly rejected after the bounds are tightened. 

{\footnotesize \cblue
    \begin{longtable}[htbp]{|r|r|r|r|r|r|r|r|r|}
    \caption{\cblue \small \bf With the spanning tree heuristic: the objective value, runtime, and relative gap of ACOTS relaxations for instances with 500 to 2,312 buses (``nf.": no feasible solution within time limit).}\label{tab: large_span_paper}\\
    \hline
          & \multicolumn{4}{c|}{Objective Value} & \multicolumn{4}{c|}{Runtime (seconds) (Relative Gap)} \\
    Test Case & \multicolumn{1}{c|}{E} & \multicolumn{1}{c|}{EC} & \multicolumn{1}{c|}{ECB} & \multicolumn{1}{c|}{ECB*} & \multicolumn{1}{c|}{E} & \multicolumn{1}{c|}{EC} & \multicolumn{1}{c|}{ECB} & \multicolumn{1}{c|}{ECB*} \\
    \hline
    case500\_goc\_sad & 455966.2 & nf.   & nf.   & nf.   & 376.13 & -     & -     & - \\
    \hline
    case588\_sdet\_sad & 309538.7 & nf.   & nf.   & nf.   & (0.6\%) & -     & -     & - \\
    \hline
    case793\_goc\_sad & 267942.0 & 268916.6 & 274171.5 & 273799.2 & (1.4\%) & (1.8\%) & (3.3\%) & (3.2\%) \\
    \hline
    case1354\_pegase\_sad & 1239689.5 & 1239701.2 & nf.   & nf.   & (0.1\%) & (0.1\%) & -     & - \\
    \hline
    case2000\_goc\_sad & nf.   & 979501.4 & nf.   & nf.   & -     & (65.1\%) & -     & - \\
    \hhline{|=|=|=|=|=|=|=|=|=|}
    case500\_goc\_api & 668613.9 & 668617.9 & nf.   & nf.   & 8774.69 & 753.22 & -     & - \\
    \hline
    case588\_sdet\_api & 389283.1 & 389286.7 & 390014.4 & 390166.5 & 3974.67 & (0.0\%) & (0.1\%) & (0.0\%) \\
    \hline
    case793\_goc\_api & nf.   & 276754.6 & 281581.3 & 283829.4 & -     & (0.3\%) & (1.6\%) & (2.4\%) \\
    \hline
    case1354\_pegase\_api & 1502166.8 & 1490268.6 & nf.   & nf.   & (0.8\%) & (0.0\%) & -     & - \\
    \hline
    case2000\_goc\_api & nf.   & 1438342.5 & nf.   &   nf.    & -     & 286.95 & -     &  - \\
    \hline
    case2312\_goc\_api & 496817.5 & nf.   &  nf.     & nf.   & (26.3\%) & -     &  -     & - \\
    \hhline{|=|=|=|=|=|=|=|=|=|}
    case500\_goc & 453852.0 & nf.   & nf.   & nf.   & 3512.71 & -     & -     & - \\
    \hline
    case588\_sdet & 307165.3 & 307164.4 & 307746.1 & 307766.5 & (0.1\%) & (0.1\%) & (0.1\%) & (0.1\%) \\
    \hline
    case793\_goc & 256772.0 & 256772.4 & 256852.8 & 256908.6 & (0.0\%) & (0.0\%) & (0.0\%) & (0.0\%) \\
    \hline
    case1354\_pegase & 1239314.7 & 1239327.2 & nf.   & nf.   & (0.1\%) & (0.1\%) & -     & - \\
    \hline
    case2000\_goc & 970466.8 & nf.   & nf.   &  nf.     & (57.8\%) & -     & -     & - \\
    \hline
    \end{longtable}}

In sum, the ``E" and ``EC" relaxations with the max spanning tree heuristic are very useful in obtaining results for larger PGLib cases. Given the critical role of warm-start solution in performance, it is promising that more large cases can be handled once the solver can effectively adopt feasible warm-start solutions.

}

\subsection{Analysis for varying load profiles}
We uniformly increase the loading condition, starting from the nominal value, of case30\_ieee instance, and observe the number of lines that are switched off. As shown in Figure \ref{fig: sens_lines}, the number of off lines first decreases, and then increases. This is because when the load is at the lower levels, some lines are redundant and are switched off to save costs. However, when the load is very high and the network is congested, lines are switched off to avoid congestion.
As suggested by \cite{fisher2008optimal}, it would be beneficial to solve the ACOTS problem frequently to obtain optimal line switching decisions for different load profiles. 
\begin{figure}[htbp]
    \centering
\includegraphics[scale=0.55]{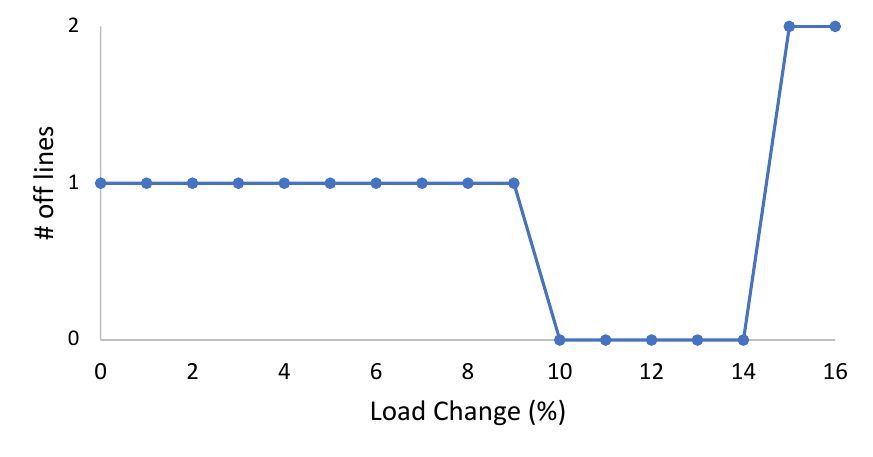} 
    \caption{Number of lines switched off with different load levels.}
    \label{fig: sens_lines}
\end{figure}
\subsection{Lifted cycle constraints for ACOPF-QC}\label{ch: cycle_acopf}
We also add the linearized \cycconstrs~(including the novel ones we derived) to the ACOPF-QC relaxation, This happens to be the special case of the ACOTS-QC with all the lines of the network {\cblue switched} on. These cycle constraints are linearized using the strong extreme-point representation, which is also new. We experiment with ``TYP", ``SAD" and ``API" cases with up to 2869 buses (102 cases in total), and observe improvements of optimality gaps in several instances. We report cases with greater than 0.03\% improvement in objectives in Table \ref{tab: opf}{\cblue, and include the results for all tested cases in Table \ref{tab: all_opf} of the Online Supplement.} The results show that, for ACOPF-QC the \cycconstrs \ are more useful in tightening ``SAD" and ``API" instances, and for smaller-size test cases.

\begin{table}[htbp]
\centering
\caption{Comparing optimality gaps (\%) for ``E" and ``EC" relaxations for ACOPF-QC. {\cblue Including all cases with improvement greater than 0.03\%}.}
\label{tab: opf}
\renewcommand{\arraystretch}{1.15}
\scalebox{1.02}{\small
\begin{tabular}{|r|r|r|}
\hline
Test Case             & \multicolumn{1}{c|}{E} & \multicolumn{1}{c|}{EC} \\ \hhline{|=|=|=|}
case3\_lmbd\_sad        & 1.38                     & 1.31                      \\ \hline
case14\_ieee\_sad       & 19.16                    & 13.10                     \\ \hline
case24\_ieee\_rts\_sad  & 2.74                     & 2.20                      \\ \hline
case57\_ieee\_sad       & 0.32                     & 0.25                      \\ \hline
case73\_ieee\_rts\_sad  & 2.37                     & 1.80                      \\ \hline
case240\_pserc\_sad     & 4.34                     & 4.24                      \\ \hline
case2383wp\_k\_sad      & 1.91                     & 1.88                      \\ \hhline{|=|=|=|}
case3\_lmbd\_api        & 4.53                     & 3.85                      \\ \hline
case24\_ieee\_rts\_api  & 11.02                    & 10.88                     \\ \hline
case73\_ieee\_rts\_api  & 9.52                     & 9.31                      \\ \hline
case179\_goc\_api       & 5.86                     & 5.75                      \\ \hline
\end{tabular}}
\end{table}

\FloatBarrier
\section{Conclusion}
In this paper, we strengthen the on/off QC relaxation of the ACOTS model by the extreme-point representation technique, several valid inequalities added via branch-and-cut, and the OBBT algorithm. {\cblue We also speed up the solution with a maximum spanning tree-based heuristic.} Experiments on PGLib instances show that the strengthened ACOTS-QC formulation significantly improves lower bounds in several instances, especially for small angle-difference instances and congested instances. {\cblue We also experiment with large-scale instances that are unexplored in the literature, demonstrating the usefulness of our relaxations and heuristic.} Our proposed \cycconstrs~improve bounds of ACOTS-QC as well as ACOPF-QC relaxations{\cblue , and are helpful for solving more large-scale instances. We believe that our methods and findings make an important contribution towards solving ACOTS problems of more practical scale.} 

Considering large-scale grids, operated in a close-to-real-time fashion, ACOTS is still a very hard problem to solve with optimality guarantees. As the line switching decisions are sensitive to load changes, it would be helpful to develop stochastic programming models that provide more robust solutions. {\cblue It would also be interesting to improve the maximum spanning tree heuristic, further reducing binary switching decisions while maintaining high accuracy. }To address scaling issues, it would be useful to (i) balance the trade-off between run time and tighter formulations{\cblue ; }(ii) develop faster decomposition-based distributed algorithms{\cblue ; and (iii) study more scalable ACOPF relaxations.}  

\ACKNOWLEDGMENT{%
The authors gratefully acknowledge funding from the U.S. Department of Energy ``Laboratory Directed Research and Development (LDRD)" program under the project ``20230091ER: Learning to Accelerate Global Solutions for Non-convex Optimization". The authors also thank Clemson University for the allocation of compute time on the Palmetto Cluster. Additionally, we thank the anonymous reviewers for their critical feedback, which improved the presentation of this paper.}

\bibliographystyle{informs2014} 
\bibliography{references}

\newpage
\begin{APPENDICES}
 \section{Additional Experiment Results for ACOTS Instances Up to 300 Buses}\label{ch:add_experiments_300}
{\cblue \footnotesize
    \begin{longtable}{|r|>{\raggedleft\arraybackslash}p{0.6cm}|>{\raggedleft\arraybackslash}p{0.6cm}|>{\raggedleft\arraybackslash}p{0.6cm}|r|r|}
    \caption{\cblue \small \bf Relative gap (\%) of ACOTS relaxations at termination (``nf.": no feasible solution within time limit).}\label{tab: rel_gap_ots}\\
    \hline
    Test Case & \multicolumn{1}{c|}{P} & \multicolumn{1}{c|}{E} & \multicolumn{1}{c|}{EC} & \multicolumn{1}{c|}{ECB} & \multicolumn{1}{c|}{ECB*} \\
    \hline
    case3\_lmbd & 0.0   & 0.0   & 0.0   & 0.0   & 0.0 \\
    \hline
    case5\_pjm & 0.0   & 0.0   & 0.0   & 0.0   & 0.0 \\
    \hline
    case14\_ieee & 0.0   & 0.0   & 0.0   & 0.0   & 0.0 \\
    \hline
    case24\_ieee\_rts & 0.0   & 0.0   & 0.0   & 0.0   & 0.0 \\
    \hline
    case30\_as & 0.0   & 0.0   & 0.0   & 0.0   & 0.0 \\
    \hline
    case30\_ieee & 0.0   & 0.0   & 0.0   & 0.0   & 0.0 \\
    \hline
    case39\_epri & 0.0   & 0.0   & 0.0   & 0.0   & 0.0 \\
    \hline
    case57\_ieee & 0.0   & 0.0   & 0.0   & 0.0   & 0.0 \\
    \hline
    case73\_ieee\_rts & 0.0   & 0.0   & 0.0   & 0.0   & 0.0 \\
    \hline
    case89\_pegase & 0.1   & 0.1   & nf.   & nf.   & nf. \\
    \hline
    case118\_ieee & 0.0   & 0.0   & 0.0   & 0.0   & 0.0 \\
    \hline
    case162\_ieee\_dtc & 0.1   & 0.1   & 0.2   & 0.6   & 0.2 \\
    \hline
    case179\_goc & 0.0   & 0.0   & 0.0   & 0.0   & 0.0 \\
    \hline
    case200\_activ & 0.0   & 0.0   & 0.0   & 0.0   & nf. \\
    \hline
    case240\_pserc & 0.1   & 0.0   & 0.0   & 0.2   & 0.0 \\
    \hline
    case300\_ieee & nf.   & 2.6   & 3.0   & 1.8   & nf. \\
    \hhline{|=|=|=|=|=|=|}
    case3\_lmbd\_sad & 0.0   & 0.0   & 0.0   & 0.0   & 0.0 \\
    \hline
    case5\_pjm\_sad & 0.0   & 0.0   & 0.0   & 0.0   & 0.0 \\
    \hline
    case14\_ieee\_sad & 0.0   & 0.0   & 0.0   & 0.0   & 0.0 \\
    \hline
    case24\_ieee\_rts\_sad & 0.0   & 0.0   & 0.0   & 0.0   & 0.0 \\
    \hline
    case30\_as\_sad & 0.0   & 0.0   & 0.0   & 0.0   & 0.0 \\
    \hline
    case30\_ieee\_sad & 0.0   & 0.0   & 0.0   & 0.0   & 0.0 \\
    \hline
    case39\_epri\_sad & 0.0   & 0.0   & 0.0   & 0.0   & 0.0 \\
    \hline
    case57\_ieee\_sad & 0.0   & 0.0   & 0.0   & 0.0   & 0.0 \\
    \hline
    case73\_ieee\_rts\_sad & 0.3   & 0.1   & 0.2   & 0.5   & 0.2 \\
    \hline
    case89\_pegase\_sad & 0.0   & 0.1   & nf.   & nf.   & nf. \\
    \hline
    case118\_ieee\_sad & 0.0   & 0.0   & 0.0   & 0.1   & 0.0 \\
    \hline
    case162\_ieee\_dtc\_sad & 0.2   & 0.3   & 0.4   & 0.6   & 0.4 \\
    \hline
    case179\_goc\_sad & nf.   & 0.0   & 0.0   & 0.1   & 0.1 \\
    \hline
    case200\_activ\_sad & 0.0   & 0.0   & 0.0   & 0.0   & nf. \\
    \hline
    case240\_pserc\_sad & nf.   & 0.3   & 0.8   & 0.4   & nf. \\
    \hline
    case300\_ieee\_sad & nf.   & 1.9   & 2.8   & 2.8   & 2.0 \\
    \hhline{|=|=|=|=|=|=|}
    case3\_lmbd\_api & 0.0   & 0.0   & 0.0   & 0.0   & 0.0 \\
    \hline
    case5\_pjm\_api & 0.0   & 0.0   & 0.0   & 0.0   & 0.0 \\
    \hline
    case14\_ieee\_api & 0.0   & 0.0   & 0.0   & 0.0   & 0.0 \\
    \hline
    case24\_ieee\_rts\_api & 0.0   & 0.0   & 0.0   & 0.0   & 0.0 \\
    \hline
    case30\_as\_api & 0.0   & 0.0   & 0.0   & 0.0   & 0.0 \\
    \hline
    case30\_ieee\_api & 0.0   & 0.0   & 0.0   & 0.0   & 0.0 \\
    \hline
    case39\_epri\_api & 0.0   & 0.0   & 0.0   & 0.0   & 0.0 \\
    \hline
    case57\_ieee\_api & 0.0   & 0.0   & 0.0   & 0.0   & 0.0 \\
    \hline
    case73\_ieee\_rts\_api & 0.0   & 0.0   & 2.7   & 0.0   & 0.0 \\
    \hline
    case89\_pegase\_api & nf.   & 0.1   & nf.   & nf.   & nf. \\
    \hline
    case118\_ieee\_api & 0.0   & 0.0   & 0.1   & 0.0   & 0.1 \\
    \hline
    case162\_ieee\_dtc\_api & 0.1   & 0.1   & 0.2   & 0.9   & 0.2 \\
    \hline
    case179\_goc\_api & 0.0   & 0.0   & 0.0   & 0.0   & 0.0 \\
    \hline
    case200\_activ\_api & 0.0   & 0.0   & 0.0   & 0.0   & 0.0 \\
    \hline
    case240\_pserc\_api & nf.   & 0.0   & 0.0   & 0.1   & 0.1 \\
    \hline
    case300\_ieee\_api & 0.0   & 0.0   & 0.0   & 0.0   & 0.0 \\
    \hline
    \end{longtable}}
{\cblue \footnotesize
    \begin{longtable}{|r|r|}
    \caption{\cblue \small \bf Separation problem building time (seconds) of ECB*.}\label{tab: sep_build}\\
    \hline
    Test Case & \multicolumn{1}{c|}{Time} \\
    \hline
    case3\_lmbd & 1.56 \\
    \hline
    case5\_pjm & 2.24 \\
    \hline
    case14\_ieee & 5.56 \\
    \hline
    case24\_ieee\_rts & 6.42 \\
    \hline
    case30\_as & 5.59 \\
    \hline
    case30\_ieee & 7.62 \\
    \hline
    case39\_epri & 8.08 \\
    \hline
    case57\_ieee & 4.66 \\
    \hline
    case73\_ieee\_rts & 6.53 \\
    \hline
    case89\_pegase & 4.73 \\
    \hline
    case118\_ieee & 8.30 \\
    \hline
    case179\_goc & 9.09 \\
    \hline
    case200\_activ & 19.22 \\
    \hhline{|=|=|}
    case3\_lmbd\_sad & 1.54 \\
    \hline
    case5\_pjm\_sad & 1.77 \\
    \hline
    case14\_ieee\_sad & 3.10 \\
    \hline
    case24\_ieee\_rts\_sad & 2.54 \\
    \hline
    case30\_as\_sad & 2.58 \\
    \hline
    case30\_ieee\_sad & 3.33 \\
    \hline
    case39\_epri\_sad & 2.65 \\
    \hline
    case57\_ieee\_sad & 3.25 \\
    \hline
    case89\_pegase\_sad & 5.30 \\
    \hline
    case118\_ieee\_sad & 6.47 \\
    \hline
    case179\_goc\_sad & 3.70 \\
    \hline
    case200\_activ\_sad & 7.04 \\
    \hhline{|=|=|}
    case3\_lmbd\_api & 0.23 \\
    \hline
    case5\_pjm\_api & 3.25 \\
    \hline
    case14\_ieee\_api & 3.18 \\
    \hline
    case24\_ieee\_rts\_api & 17.61 \\
    \hline
    case30\_as\_api & 3.59 \\
    \hline
    case30\_ieee\_api & 3.57 \\
    \hline
    case39\_epri\_api & 6.88 \\
    \hline
    case57\_ieee\_api & 7.42 \\
    \hline
    case73\_ieee\_rts\_api & 8.83 \\
    \hline
    case89\_pegase\_api & 6.21 \\
    \hline
    case118\_ieee\_api & 5.81 \\
    \hline
    case162\_ieee\_dtc\_api & 7.03 \\
    \hline
    case179\_goc\_api & 13.66 \\
    \hline
    case200\_activ\_api & 20.32 \\
    \hline
    case240\_pserc\_api & 5.96 \\
    \hline
    case300\_ieee\_api & 7.03 \\
    \hline
    \end{longtable}}

 
  {\cblue \footnotesize
    \begin{longtable}[htbp]{|r|r|r|r|r|r|}      
      \caption{\cblue \small \bf Optimality gap (\%) of ACOTS relaxations at the root node (bold numbers: improved gaps after tightening the relaxation with ``E", ``EC" and ``ECB"; bold-italic numbers: ``ECB*" optimality gaps can be improved with ``ECB"). }\label{tab: root_relax}\\
    \hline
    Instance & \multicolumn{1}{c|}{P} & \multicolumn{1}{c|}{E} & \multicolumn{1}{c|}{EC} & \multicolumn{1}{c|}{ECB} & \multicolumn{1}{c|}{ECB*} \\
    \hline
    case3\_lmbd & 3.9   & 3.9   & \textbf{3.8} & \textbf{2.8} & \textit{\textbf{3.1}} \\
    \hline
    case5\_pjm & 2.4   & 2.4   & 2.4   & 2.4   & 2.4 \\
    \hline
    case14\_ieee & 5.9   & 5.9   & 5.9   & \textbf{5.8} & 5.8 \\
    \hline
    case24\_ieee\_rts & 4.1   & 4.1   & 4.1   & 4.1   & 4.1 \\
    \hline
    case30\_as & 4.5   & 4.5   & 4.5   & 4.5   & 4.5 \\
    \hline
    case30\_ieee & 25.8  & 25.8  & 25.8  & \textbf{25.1} & 25.1 \\
    \hline
    case39\_epri & 1.2   & 1.2   & 1.2   & 1.2   & 1.2 \\
    \hline
    case57\_ieee & 7.6   & 7.6   & 7.6   & 7.6   & 7.6 \\
    \hline
    case73\_ieee\_rts & 4.0   & 4.0   & 4.0   & 4.0   & 4.0 \\
    \hline
    case89\_pegase & 1.8   & 1.8   & 1.8   & 1.8      & 1.8 \\
    \hline
    case118\_ieee & 4.1   & 4.1   & 4.2   & 4.2   & 4.1 \\
    \hline
    case179\_goc & 0.5   & 0.5   & 0.5   & 0.5   & 0.5 \\
    \hline
    case200\_activ & 0.3   & 0.3   & 0.3   & 0.3   & 0.3 \\\hline
    \hhline{|=|=|=|=|=|=|}
    case3\_lmbd\_sad & 6.1   & \textbf{5.8} & \textbf{5.1} & \textbf{1.6} & 1.6 \\
    \hline
    case5\_pjm\_sad & 4.5   & 4.5   & \textbf{2.1} & \textbf{1.1} & \textit{\textbf{3.0}} \\
    \hline
    case14\_ieee\_sad & 24.8  & 24.8  & 24.8  & \textbf{9.5} & \textit{\textbf{12.5}} \\
    \hline
    case24\_ieee\_rts\_sad & 10.8  & 10.8  & \textbf{10.4} & \textbf{7.7} & \textit{\textbf{10.4}} \\
    \hline
    case30\_as\_sad & 12.1  & 12.1  & \textbf{11.6} & \textbf{9.7} & \textit{\textbf{11.9}} \\
    \hline
    case30\_ieee\_sad & 21.2  & 21.2  & 21.2  & \textbf{5.0} & \textit{\textbf{5.9}} \\
    \hline
    case39\_epri\_sad & 1.5   & 1.5   & 1.5   & \textbf{1.4} & 1.4 \\
    \hline
    case57\_ieee\_sad & 2.5   & 2.5   & 2.5   & \textbf{2.4} & \textit{\textbf{2.5}} \\
    \hline
    case89\_pegase\_sad & 1.7   & 1.7   & 1.7   & 1.7      & 1.7 \\
    \hline
    case118\_ieee\_sad & 5.0   & 5.0   & 5.0   & 5.0   & 5.0 \\
    \hline
    case179\_goc\_sad & 0.5   & 0.5   & 0.5   & 0.5   & 0.5 \\
    \hline
    case200\_activ\_sad & 0.3   & 0.3   & 0.3   & 0.3   & 0.3 \\
    \hhline{|=|=|=|=|=|=|}
    case3\_lmbd\_api & 6.6   & \textbf{6.5} & \textbf{6.1} & \textbf{0.4} & \textit{\textbf{7.0}} \\
    \hline
    case5\_pjm\_api & 2.9   & 2.9   & 2.9   & \textbf{1.2} & 1.2 \\
    \hline
    case14\_ieee\_api & 22.1  & 22.1  & 22.1  & \textbf{11.3} & \textit{\textbf{16.2}} \\
    \hline
    case24\_ieee\_rts\_api & 17.2  & 17.2  & 17.2  & \textbf{14.0} & \textit{\textbf{16.3}} \\
    \hline
    case30\_as\_api & 10.6  & 10.6  & 10.6  & \textbf{10.2} & 10.2 \\
    \hline
    case30\_ieee\_api & 15.1  & 15.1  & 15.1  & \textbf{5.7} & \textit{\textbf{7.5}} \\
    \hline
    case39\_epri\_api & 2.2   & 2.2   & 2.2   & \textbf{2.1} & 2.1 \\
    \hline
    case57\_ieee\_api & 3.9   & 3.9   & 3.9   & 3.9   & 3.9 \\
    \hline
    case73\_ieee\_rts\_api & 11.9  & 11.9  & 11.9  & \textbf{11.3} & 11.3 \\
    \hline
    case89\_pegase\_api & 2.4   & 2.5   & 2.5   & 2.5      & 2.5 \\
    \hline
    case118\_ieee\_api & 13.7  & \textbf{13.3} & 13.3  &\textbf{13.2}       & 13.2 \\
    \hline
    case162\_ieee\_dtc\_api & 6.7   & 6.7   & 6.8   &  \textbf{6.7}     & 6.7 \\
    \hline
    case179\_goc\_api & 11.7  & 11.7  & 11.7  &  \textbf{1.7}     & 1.7 \\
    \hline
    case200\_activ\_api & 2.0   & 2.0   & 2.0   &2.0       & 2.0 \\
    \hline
    case240\_pserc\_api & 2.0   & 2.1   & 2.1   & 2.1    & 2.1 \\
    \hline
    case300\_ieee\_api & 4.6   & 4.6   & 4.6   & 4.6      & 4.4 \\
    \hline
    \end{longtable}}

{\scriptsize \cblue
\begin{longtable}{| r | >{\raggedleft\arraybackslash}p{1.3cm} | >{\raggedleft\arraybackslash}p{1.3cm} | r | r |}        
\caption{\cblue \small \bf Optimality gap and runtime of ACOTS relaxations ``PB" and ``EB" (bold-italic numbers: optimality gaps can be improved after tightening the relaxation; ``tl.": hits the time limit).}\label{tab: eb_appendix}\\
\hline
 & \multicolumn{2}{l|}{Optimality Gap (\%)} & \multicolumn{2}{l|}{Runtime (seconds)} \\
\multicolumn{1}{|r|}{Test Case} & \multicolumn{1}{c|}{PB} & \multicolumn{1}{c|}{EB} & \multicolumn{1}{c|}{PB} & \multicolumn{1}{c|}{EB}\\ \hline
case3\_lmbd & \textbf{\emph{0.2}} & \textbf{\emph{0.1}} & 0.03 & 0.03 \\ \hline
case5\_pjm & \textbf{\emph{1.2}} & 1.1 & 0.04 & 0.05 \\ \hline
case14\_ieee & 0.1 & 0.1 & 0.28 & 0.42 \\ \hline
case24\_ieee\_rts & 0.0 & 0.0 & 0.74 & 1.70 \\ \hline
case30\_as& 0.1 & 0.1 & 1.63 & 11.93 \\ \hline
case30\_ieee & \textbf{\emph{11.7}} & 11.0 & 1.20 & 2.36 \\ \hline
case39\_epri & \textbf{\emph{0.1}} & 0.0 & 0.36 & 0.62 \\ \hline
case57\_ieee & 0.1 & 0.1 & 12.14 & 24.97 \\ \hline
case73\_ieee\_rts & 0.0 & 0.0 & 14.94 & 90.88 \\ \hline
case89\_pegase & 0.1 & 0.1 & tl. & tl. \\ \hline
case118\_ieee & 0.3 & 0.3 & 263.30 & 679.13 \\ \hline
case179\_goc & 0.1 & 0.1 & 88.37 & 180.36 \\ \hline
case200\_activ & 0.0 & 0.0 & 790.71 & 4200.74 \\\hhline{|=|=|=|=|=|}
case3\_lmbd\_sad & \textbf{\emph{0.3}} & \textbf{\emph{0.2}} & 0.01 & 0.02 \\ \hline
case5\_pjm\_sad & \textbf{\emph{0.4}} & \textbf{\emph{0.3}} & 0.03 & 0.04 \\ \hline
case14\_ieee\_sad & \textbf{\emph{2.3}} & \textbf{\emph{0.8}} & 0.25 & 0.51 \\ \hline
case24\_ieee\_rts\_sad& \textbf{\emph{2.3}} & \textbf{\emph{0.9}} & 3.36 & 11.33 \\ \hline
case30\_as\_sad & \textbf{\emph{1.9}} & 1.1 & 2.94 & 16.89 \\ \hline
case30\_ieee\_sad& \textbf{\emph{0.5}} & \textbf{\emph{0.2}} & 0.92 & 2.02 \\ \hline
case39\_epri\_sad & 0.1 & 0.1 & 6.29 & 12.47 \\ \hline
case57\_ieee\_sad & \textbf{\emph{0.2}} & 0.1 & 13.12 & 42.12 \\ \hline
case89\_pegase\_sad & 0.1 & 0.1 & tl. & tl. \\ \hline
case118\_ieee\_sad & 0.9 & 0.9 & 2297.54 & tl. \\ \hline
case179\_goc\_sad & 0.1 & 0.1 & tl. & tl. \\ \hline
case200\_activ\_sad & 0.0 & 0.0 & 479.80 & 1910.26 \\\hhline{|=|=|=|=|=|}
case3\_lmbd\_api & 0.0 & 0.0 & 0.01 & 0.01 \\ \hline
case5\_pjm\_api & \textbf{\emph{0.7}} & \textbf{\emph{0.4}} & 0.04 & 0.05 \\ \hline
case14\_ieee\_api & \textbf{\emph{3.4}} & \textbf{\emph{1.1}} & 0.14 & 0.18 \\ \hline
case24\_ieee\_rts\_api & \textbf{\emph{4.0}} & 1.2 & 1.13 & 3.96 \\ \hline
case30\_as\_api & 5.3 & 5.3 & 1.22 & 8.01 \\ \hline
case30\_ieee\_api & \textbf{\emph{0.8}} & \textbf{\emph{0.4}} & 0.61 & 1.04 \\ \hline
case39\_epri\_api & \textbf{\emph{0.5}} & 0.4 & 0.65 & 1.82 \\ \hline
case57\_ieee\_api & \textbf{\emph{0.1}} & 0.0 & 13.73 & 27.64 \\ \hline
case73\_ieee\_rts\_api & \textbf{\emph{2.9}} & 1.2 & 64.52 & 107.39 \\ \hline
case89\_pegase\_api & 0.2 & 0.2 & tl. & tl. \\ \hline
case118\_ieee\_api & \textbf{\emph{6.3}} & 6.1 & 1056.73 & 5482.16 \\ \hline
case162\_ieee\_dtc\_api & 1.0 & \textbf{\emph{1.0}} & tl. & tl. \\ \hline
case179\_goc\_api & 0.6 & \textbf{\emph{0.6}} & 129.72 & 218.81 \\ \hline
case200\_activ\_api & 0.0 & 0.0 & 769.25 & 884.19 \\ \hline
case240\_pserc\_api & 0.6 & 0.6 & tl. & tl. \\ \hline
case300\_ieee\_api & 0.7 & 0.7 & tl. & tl. \\\hline
\end{longtable}}

  {\cblue \footnotesize 
    \begin{longtable}[htbp]{|r|r|r|} 
      \caption{\cblue \small \bf Locally optimal objective values of the non-convex ACOTS model \eqref{eq:ACOTS} obtained by Juniper and Knitro (bold numbers: better objective values; ``nf.": no feasible solution within time limit).}\label{tab: junper_knitro}\\
    \hline
    Test Case & Juniper & \multicolumn{1}{c|}{Knitro} \\
    \hline
    case3\_lmbd & 5812.6 & 5812.6 \\
    \hline
    case5\_pjm & 15174.0 & 15174.0 \\
    \hline
    case14\_ieee & 2178.1 & 2178.1 \\
    \hline
    case24\_ieee\_rts & 63352.2 & 63352.2 \\
    \hline
    case30\_as & 803.1 & 803.1 \\
    \hline
    case30\_ieee & 7579.0 & 7579.0 \\
    \hline
    case39\_epri & \textbf{137728.7} & 137746.4 \\
    \hline
    case57\_ieee & \textbf{37559.3} & 37559.5 \\
    \hline
    case73\_ieee\_rts & 189764.1 & 189764.1 \\
    \hline
    case89\_pegase & 106656.2 & \textbf{106623.0} \\
    \hline
    case118\_ieee & 96739.8 & \textbf{96730.7} \\
    \hline
    case179\_goc & nf.   & \textbf{754083.7} \\
    \hline
    case200\_activ & \textbf{27557.6} & 27558.3 \\
    \hhline{|=|=|=|}
    case3\_lmbd\_sad & 5959.3 & 5959.3 \\
    \hline
    case5\_pjm\_sad & 26108.8 & 26108.8 \\
    \hline
    case14\_ieee\_sad & \textbf{2727.5} & 2768.5 \\
    \hline
    case24\_ieee\_rts\_sad & \textbf{75794.0} & 76613.0 \\
    \hline
    case30\_as\_sad & \textbf{893.9} & 894.3 \\
    \hline
    case30\_ieee\_sad & 8188.6 & 8188.6 \\
    \hline
    case39\_epri\_sad & \textbf{147472.9} & 147908.9 \\
    \hline
    case57\_ieee\_sad & 38597.8 & 38597.8 \\
    \hline
    case89\_pegase\_sad & nf.   & \textbf{106633.6} \\
    \hline
    case118\_ieee\_sad & \textbf{97572.5} & 100429.9 \\
    \hline
    case179\_goc\_sad & nf.   & \textbf{762532.5} \\
    \hline
    case200\_activ\_sad & \textbf{27557.6} & 27558.3 \\
    \hhline{|=|=|=|} 
    case3\_lmbd\_api & 10636.0 & 10636.0 \\
    \hline
    case5\_pjm\_api & 75190.3 & 75190.3 \\
    \hline
    case14\_ieee\_api & 5999.4 & 5999.4 \\
    \hline
    case24\_ieee\_rts\_api & 119743.1 & 119743.1 \\
    \hline
    case30\_as\_api & 3065.8 & \textbf{2925.1} \\
    \hline
    case30\_ieee\_api & 17936.5 & 17936.5 \\
    \hline
    case39\_epri\_api & 246723.0 & 246723.0 \\
    \hline
    case57\_ieee\_api & \textbf{49271.9} & 49277.9 \\
    \hline
    case73\_ieee\_rts\_api & \textbf{385277.3} & 387328.1 \\
    \hline
    case89\_pegase\_api & \textbf{100325.3} & 123031.4 \\
    \hline
    case118\_ieee\_api & \textbf{181535.8} & 210056.6 \\
    \hline
    case162\_ieee\_dtc\_api & \textbf{116923.8} & 118378.1 \\
    \hline
    case179\_goc\_api & nf.   & \textbf{1932020.5} \\
    \hline
    case200\_activ\_api & 35701.5 & \textbf{35700.8} \\
    \hline
    case240\_pserc\_api & 4640000.0 & \textbf{4638308.7} \\
    \hline
    case300\_ieee\_api & nf.   & \textbf{684990.4} \\
    \hline
    \end{longtable}}

\section{Results for All ACOTS Instances with 500 to 2312 Buses}\label{ch:add_experiments_large}

{\cblue \footnotesize
    \begin{longtable}[htbp]{|r|r|r|r|r|>{\raggedleft\arraybackslash}p{1.4cm}|>{\raggedleft\arraybackslash}p{1.4cm}|>{\raggedleft\arraybackslash}p{1.4cm}|}
    \caption{\cblue \small \bf Objective value, runtime, and relative gap of ACOTS relaxations for all instances with 500 to 2312 buses (``nf.": no feasible solution within time limit).}\label{tab: large_inst}\\
    \hline
          &       & \multicolumn{3}{c|}{Objective Value} & \multicolumn{3}{c|}{Runtime (seconds) (Relative Gap)} \\
    Test Case & \multicolumn{1}{c|}{UB} & \multicolumn{1}{c|}{P} & \multicolumn{1}{c|}{E} & \multicolumn{1}{c|}{EC} & \multicolumn{1}{c|}{P} & \multicolumn{1}{c|}{E} & \multicolumn{1}{c|}{EC} \\
    \hline
    case500\_goc\_sad & 487397.2 & nf.   & 455966.2 & nf.   & -     & (0.5\%) & - \\
    \hline
    case588\_sdet\_sad & 313797.2 & nf.   & nf.   & nf.   & -     & -     & - \\
    \hline
    case793\_goc\_sad & 285798.4 & nf.   & nf.   & nf.   & -     & -     & - \\
    \hline
    case1354\_pegase\_sad & 1258848.0 & nf.   & 1239689.5 & 1239701.2 & -     & (0.1\%) & (0.2\%) \\
    \hline
    case1888\_rte\_sad & 1413922.9 & nf.   & nf.   & nf.   & -     & -     & - \\
    \hline
    case1951\_rte\_sad & 2089370.1 & nf.   & nf.   & nf.   & -     & -     & - \\
    \hline
    case2000\_goc\_sad & 992879.8 & nf.   & nf.   & 979501.4 & -     & -     & (65.1\%) \\
    \hline
    case2312\_goc\_sad & 461748.3 & nf.   & nf.   & nf.   & -     & -     & - \\
    \hhline{|=|=|=|=|=|=|=|=|}
    case500\_goc\_api & 675967.9 & nf.   & 668613.9 & 668617.9 & -     & 2332.84 & 9289.02 \\
    \hline
    case588\_sdet\_api & 391951.0 & 389971.1 & nf.   & nf.   & (0.2\%) & -     & - \\
    \hline
    case793\_goc\_api & 318853.5 & nf.   & nf.   & nf.   & -     & -     & - \\
    \hline
    case1354\_pegase\_api & 1498271.0 & nf.   & nf.   & 1490268.6 & -     & -     & (0.1\%) \\
    \hline
    case1888\_rte\_api & 1951676.1 & nf.   & nf.   & nf.   & -     & -     & - \\
    \hline
    case1951\_rte\_api & 2410827.1 & nf.   & nf.   & nf.   & -     & -     & - \\
    \hline
    case2000\_goc\_api & 1468630.3 & nf.   & nf.   & 1438342.5 & -     & -     & (86.2\%) \\
    \hline
    case2312\_goc\_api & 571446.1 & nf.   & 496817.5 & nf.   & -     & (26.3\%) & - \\
    \hhline{|=|=|=|=|=|=|=|=|}
    case500\_goc & 454946.0 & nf.   & 453852.0 & nf.   & -     & 8897.06 & - \\
    \hline
    case588\_sdet & 310016.0 & 307339.6 & nf.   & nf.   & (0.1\%) & -     & - \\
    \hline
    case793\_goc & 259097.8 & 256776.8 & nf.   & nf.   & (0.0\%) & -     & - \\
    \hline
    case1354\_pegase & 1258844.0 & nf.   & 1239314.7 & 1239327.2 & -     & 1770.48 & (0.1\%) \\
    \hline
    case1888\_rte & 1375888.3 & nf.   & nf.   & nf.   & -     & -     & - \\
    \hline
    case1951\_rte & 2085530.3 & nf.   & 2083344.3 & nf.   & -     & (0.1\%) & - \\
    \hline
    case2000\_goc & 973432.5 & nf.   & 970466.8 & nf.   & -     & (57.8\%) & - \\
    \hline
    case2312\_goc & 441330.3 & nf.   & nf.   & nf.   & -     & -     & - \\
    \hline
    \end{longtable}}

{\footnotesize \cblue
    \begin{longtable}[htbp]{|r|r|r|r|r|r|r|r|r|}
    \caption{\cblue \small \bf With the spanning tree heuristic: the objective value, runtime, and relative gap of ACOTS relaxations for all instances with 500 to 2312 buses (``nf.": no feasible solution within time limit).}\label{tab: large_span}\\
    \hline
          & \multicolumn{4}{c|}{Objective Value} & \multicolumn{4}{c|}{Runtime (seconds) (Relative Gap)} \\
    Test Case & \multicolumn{1}{c|}{E} & \multicolumn{1}{c|}{EC} & \multicolumn{1}{c|}{ECB} & \multicolumn{1}{c|}{ECB*} & \multicolumn{1}{c|}{E} & \multicolumn{1}{c|}{EC} & \multicolumn{1}{c|}{ECB} & \multicolumn{1}{c|}{ECB*} \\
    \hline
    case500\_goc\_sad & 455966.2 & nf.   & nf.   & nf.   & 376.13 & -     & -     & - \\
    \hline
    case588\_sdet\_sad & 309538.7 & nf.   & nf.   & nf.   & (0.6\%) & -     & -     & - \\
    \hline
    case793\_goc\_sad & 267942.0 & 268916.6 & 274171.5 & 273799.2 & (1.4\%) & (1.8\%) & (3.3\%) & (3.2\%) \\
    \hline
    case1354\_pegase\_sad & 1239689.5 & 1239701.2 & nf.   & nf.   & (0.1\%) & (0.1\%) & -     & - \\
    \hline
    case1888\_rte\_sad & nf.   & nf.   &  nf.     & nf.   & -     & -     &   -    & - \\
    \hline
    case1951\_rte\_sad & nf.   & nf.   & nf.   & nf.   & -     & -     & -     & - \\
    \hline
    case2000\_goc\_sad & nf.   & 979501.4 & nf.   & nf.   & -     & (65.1\%) & -     & - \\
    \hline
    case2312\_goc\_sad & nf.   & nf.   & nf.   & nf.   & -     & -     &   -    & - \\
    \hhline{|=|=|=|=|=|=|=|=|=|}
    case500\_goc\_api & 668613.9 & 668617.9 & nf.   & nf.   & 8774.69 & 753.22 & -     & - \\
    \hline
    case588\_sdet\_api & 389283.1 & 389286.7 & 390014.4 & 390166.5 & 3974.67 & (0.0\%) & (0.1\%) & (0.0\%) \\
    \hline
    case793\_goc\_api & nf.   & 276754.6 & 281581.3 & 283829.4 & -     & (0.3\%) & (1.6\%) & (2.4\%) \\
    \hline
    case1354\_pegase\_api & 1502166.8 & 1490268.6 & nf.   & nf.   & (0.8\%) & (0.0\%) & -     & - \\
    \hline
    case1888\_rte\_api & nf.   & nf.   &   nf.    & nf.   & -     & -     &  -     & - \\
    \hline
    case1951\_rte\_api & nf.   & nf.   & nf.   & nf.   & -     & -     & -     & - \\
    \hline
    case2000\_goc\_api & nf.   & 1438342.5 & nf.   &   nf.    & -     & 286.95 & -     & - \\
    \hline
    case2312\_goc\_api & 496817.5 & nf.   &  nf.     & nf.   & (26.3\%) & -     &  -     & - \\
    \hhline{|=|=|=|=|=|=|=|=|=|}
    case500\_goc & 453852.0 & nf.   & nf.   & nf.   & 3512.71 & -     & -     & - \\
    \hline
    case588\_sdet & 307165.3 & 307164.4 & 307746.1 & 307766.5 & (0.1\%) & (0.1\%) & (0.1\%) & (0.1\%) \\
    \hline
    case793\_goc & 256772.0 & 256772.4 & 256852.8 & 256908.6 & (0.0\%) & (0.0\%) & (0.0\%) & (0.0\%) \\
    \hline
    case1354\_pegase & 1239314.7 & 1239327.2 & nf.   & nf.   & (0.1\%) & (0.1\%) & -     & - \\
    \hline
    case1888\_rte & nf.   & nf.   &  nf.     & nf.   & -     & -     &   -    & - \\
    \hline
    case1951\_rte & nf.   & nf.   & nf.   & nf.   & -     & -     & -     & - \\
    \hline
    case2000\_goc & 970466.8 & nf.   & nf.   &  nf.     & (57.8\%) & -     & -     &  - \\
    \hline
    case2312\_goc & nf.   & nf.   &  nf.     & nf.   & -     & -     &  -     & - \\
    \hline
    \end{longtable}}

\section{Experiment Results for All ACOPF Instances}\label{ch:add_experiments_opf}

    {\cblue \footnotesize
    \begin{longtable}{|r|r|r|r|}
    \caption{\cblue \small \bf Comparing optimality gaps (\%) for ``E" and ``EC" relaxations for all ACOPF-QC instances (Diff.= gap of ``E"- gap of ``EC").}\label{tab: all_opf}\\
    \hline
    Test Case & \multicolumn{1}{c|}{E} & \multicolumn{1}{c|}{EC} & \multicolumn{1}{c|}{Diff.} \\
    \hline
    case3\_lmbd\_sad & 1.38  & 1.31  & 0.07 \\
    \hline
    case5\_pjm\_sad & 0.63  & 0.62  & 0.01 \\
    \hline
    case14\_ieee\_sad & 19.16 & 13.10 & 6.06 \\
    \hline
    case24\_ieee\_rts\_sad & 2.74  & 2.20  & 0.53 \\
    \hline
    case30\_as\_sad & 2.30  & 2.26  & 0.04 \\
    \hline
    case30\_ieee\_sad & 5.66  & 5.66  & 0.00 \\
    \hline
    case39\_epri\_sad & 0.20  & 0.20  & 0.00 \\
    \hline
    case57\_ieee\_sad & 0.32  & 0.25  & 0.07 \\
    \hline
    case73\_ieee\_rts\_sad & 2.37  & 1.80  & 0.57 \\
    \hline
    case89\_pegase\_sad & 0.70  & 0.70  & 0.00 \\
    \hline
    case118\_ieee\_sad & 6.63  & 6.61  & 0.02 \\
    \hline
    case162\_ieee\_dtc\_sad & 6.21  & 6.20  & 0.01 \\
    \hline
    case179\_goc\_sad & 0.99  & 0.97  & 0.02 \\
    \hline
    case200\_activ\_sad & 0.00  & 0.00  & 0.00 \\
    \hline
    case240\_pserc\_sad & 4.34  & 4.24  & 0.10 \\
    \hline
    case300\_ieee\_sad & 2.34  & 2.32  & 0.02 \\
    \hline
    case500\_goc\_sad & 6.45  & 6.44  & 0.01 \\
    \hline
    case588\_sdet\_sad & 5.94  & 5.94  & 0.00 \\
    \hline
    case793\_goc\_sad & 6.22  & 6.22  & 0.00 \\
    \hline
    case1354\_pegase\_sad & 1.52  & 1.52  & 0.00 \\
    \hline
    case1888\_rte\_sad & 2.80  & 2.80  & 0.00 \\
    \hline
    case1951\_rte\_sad & 0.42  & 0.42  & 0.00 \\
    \hline
    case2000\_goc\_sad & 1.35  & 1.35  & 0.00 \\
    \hline
    case2312\_goc\_sad & 3.39  & 3.37  & 0.02 \\
    \hline
    case2383wp\_k\_sad & 1.91  & 1.88  & 0.03 \\
    \hline
    case2736sp\_k\_sad & 1.31  & 1.31  & 0.00 \\
    \hline
    case2737sop\_k\_sad & 1.69  & 1.69  & 0.00 \\
    \hline
    case2742\_goc\_sad & 1.33  & 1.33  & 0.00 \\
    \hline
    case2746wop\_k\_sad & 1.96  & 1.96  & 0.00 \\
    \hline
    case2746wp\_k\_sad & 1.63  & 1.63  & 0.00 \\
    \hline
    case2848\_rte\_sad & 0.23  & 0.23  & 0.00 \\
    \hline
    case2853\_sdet\_sad & 1.64  & 1.64  & 0.00 \\
    \hline
    case2868\_rte\_sad & 0.55  & 0.55  & 0.00 \\
    \hline
    case2869\_pegase\_sad & 1.01  & 1.01  & 0.00 \\
    \hhline{|=|=|=|=|}
    case3\_lmbd\_api & 4.53  & 3.85  & 0.68 \\
    \hline
    case5\_pjm\_api & 4.09  & 4.09  & 0.00 \\
    \hline
    case14\_ieee\_api & 5.13  & 5.13  & 0.00 \\
    \hline
    case24\_ieee\_rts\_api & 11.02 & 10.88 & 0.14 \\
    \hline
    case30\_as\_api & 44.60 & 44.60 & 0.00 \\
    \hline
    case30\_ieee\_api & 5.45  & 5.45  & 0.00 \\
    \hline
    case39\_epri\_api & 1.69  & 1.69  & 0.00 \\
    \hline
    case57\_ieee\_api & 0.08  & 0.08  & 0.00 \\
    \hline
    case73\_ieee\_rts\_api & 9.52  & 9.31  & 0.22 \\
    \hline
    case89\_pegase\_api & 23.06 & 23.06 & 0.00 \\
    \hline
    case118\_ieee\_api & 29.61 & 29.60 & 0.01 \\
    \hline
    case162\_ieee\_dtc\_api & 4.32  & 4.32  & 0.01 \\
    \hline
    case179\_goc\_api & 5.86  & 5.75  & 0.11 \\
    \hline
    case200\_activ\_api & 0.02  & 0.02  & 0.00 \\
    \hline
    case240\_pserc\_api & 0.62  & 0.62  & 0.00 \\
    \hline
    case300\_ieee\_api & 0.82  & 0.81  & 0.00 \\
    \hline
    case500\_goc\_api & 3.44  & 3.44  & 0.00 \\
    \hline
    case588\_sdet\_api & 1.39  & 1.39  & 0.00 \\
    \hline
    case793\_goc\_api & 13.19 & 13.19 & 0.00 \\
    \hline
    case1354\_pegase\_api & 0.54  & 0.53  & 0.00 \\
    \hline
    case1888\_rte\_api & 0.22  & 0.22  & 0.00 \\
    \hline
    case1951\_rte\_api & 0.51  & 0.50  & 0.00 \\
    \hline
    case2000\_goc\_api & 2.06  & 2.06  & 0.00 \\
    \hline
    case2312\_goc\_api & 13.07 & 13.07 & 0.00 \\
    \hline
    case2383wp\_k\_api & 0.00  & 0.00  & 0.00 \\
    \hline
    case2736sp\_k\_api & 10.83 & 10.82 & 0.00 \\
    \hline
    case2737sop\_k\_api & 5.88  & 5.88  & 0.00 \\
    \hline
    case2742\_goc\_api & 24.40 & 24.40 & 0.00 \\
    \hline
    case2746wop\_k\_api & 0.00  & 0.00  & 0.00 \\
    \hline
    case2746wp\_k\_api & 0.00  & 0.00  & 0.00 \\
    \hline
    case2848\_rte\_api & 0.25  & 0.25  & 0.00 \\
    \hline
    case2853\_sdet\_api & 1.91  & 1.91  & 0.00 \\
    \hline
    case2868\_rte\_api & 0.17  & 0.17  & 0.00 \\
    \hline
    case2869\_pegase\_api & 0.99  & 0.98  & 0.00 \\
    \hhline{|=|=|=|=|}
    case3\_lmbd & 0.97  & 0.97  & 0.00 \\
    \hline
    case5\_pjm & 14.54 & 14.53 & 0.01 \\
    \hline
    case14\_ieee & 0.11  & 0.11  & 0.00 \\
    \hline
    case24\_ieee\_rts & 0.01  & 0.01  & 0.00 \\
    \hline
    case30\_as & 0.06  & 0.06  & 0.00 \\
    \hline
    case30\_ieee & 18.67 & 18.67 & 0.00 \\
    \hline
    case39\_epri & 0.54  & 0.54  & 0.00 \\
    \hline
    case57\_ieee & 0.16  & 0.16  & 0.00 \\
    \hline
    case73\_ieee\_rts & 0.03  & 0.03  & 0.00 \\
    \hline
    case89\_pegase & 0.75  & 0.74  & 0.00 \\
    \hline
    case118\_ieee & 0.77  & 0.76  & 0.01 \\
    \hline
    case162\_ieee\_dtc & 5.84  & 5.83  & 0.00 \\
    \hline
    case179\_goc & 0.15  & 0.15  & 0.00 \\
    \hline
    case200\_activ & 0.00  & 0.00  & 0.00 \\
    \hline
    case240\_pserc & 2.72  & 2.72  & 0.00 \\
    \hline
    case300\_ieee & 2.56  & 2.55  & 0.00 \\
    \hline
    case500\_goc & 0.24  & 0.24  & 0.00 \\
    \hline
    case588\_sdet & 1.91  & 1.91  & 0.00 \\
    \hline
    case793\_goc & 1.32  & 1.32  & 0.00 \\
    \hline
    case1354\_pegase & 1.55  & 1.55  & 0.00 \\
    \hline
    case1888\_rte & 2.04  & 2.04  & 0.00 \\
    \hline
    case1951\_rte & 0.13  & 0.13  & 0.00 \\
    \hline
    case2000\_goc & 0.30  & 0.30  & 0.00 \\
    \hline
    case2312\_goc & 1.89  & 1.89  & 0.00 \\
    \hline
    case2383wp\_k & 0.96  & 0.96  & 0.00 \\
    \hline
    case2736sp\_k & 0.30  & 0.29  & 0.00 \\
    \hline
    case2737sop\_k & 0.26  & 0.25  & 0.01 \\
    \hline
    case2742\_goc & 1.32  & 1.32  & 0.00 \\
    \hline
    case2746wop\_k & 0.35  & 0.35  & 0.00 \\
    \hline
    case2746wp\_k & 0.31  & 0.31  & 0.00 \\
    \hline
    case2848\_rte & 0.12  & 0.12  & 0.00 \\
    \hline
    case2853\_sdet & 0.86  & 0.86  & 0.00 \\
    \hline
    case2868\_rte & 0.09  & 0.09  & 0.00 \\
    \hline
    case2869\_pegase & 1.00  & 1.00  & 0.00 \\
    \hline
    \end{longtable}}

\FloatBarrier

\end{APPENDICES}

\end{document}